\documentclass{article}
\usepackage[utf8]{inputenc}

\usepackage[title]{appendix}
\usepackage{authblk}
\usepackage[margin=1in,vmargin=1in]{geometry}
\usepackage{amsmath,amsfonts,amsthm,amscd,amssymb,graphicx}
\usepackage{mathtools}
\usepackage{paralist}
\usepackage{epsfig} 
\usepackage{xcolor}
\usepackage{graphicx}  
\usepackage{epstopdf}
\usepackage{multicol}
\usepackage{comment}
\usepackage{float}
\usepackage[caption = false]{subfig}
\usepackage{enumitem}

 \usepackage[colorlinks=true,pdfencoding=auto, psdextra]{hyperref}
 \usepackage{soul}

\usepackage[compact]{titlesec}        
\titlespacing{\section}{0pt}{12pt}{0pt}
\AtBeginDocument{
  \setlength\abovedisplayskip{0pt}
  \setlength\belowdisplayskip{0pt}}

\def\XXint#1#2#3{{\setbox0=\hbox{$#1{#2#3}{\int}$}
\vcenter{\hbox{$#2#3$}}\kern-.5\wd0}}

\newcommand{\R}{\mathbb{R}}

\newcommand{\be}{\begin{equation}}
\newcommand{\eq}[2]{\begin{equation}\begin{split}#1\end{split}#2\end{equation}}
\newcommand{\ee}{\end{equation}}

\newcommand{\grad}{\triangledown}

\newcommand{\mat}[1]{\begin{pmatrix}#1\end{pmatrix}}
\newcommand{\imply}[1]{\Rightarrow}

\newcommand{\eps}{\varepsilon}

\newcommand{\cmt}[1]{}

\newtheorem{theorem}{Theorem}[section]

\newtheorem{lemma}[theorem]{Lemma}

\theoremstyle{definition}
\newtheorem{definition}[theorem]{Definition}
\newtheorem{remark}{Remark}

\title{
A Dynamical Systems Approach for Most Probable Escape Paths over Periodic Boundaries
}

\providecommand{\keywords}[1]{\textbf{Keywords:} #1}

\author[1]{Emmanuel Fleurantin}
\author[2]{Katherine Slyman}
\author[3]{Blake Barker}
\author[4]{Christopher K.R.T. Jones}
\affil[1,4]{\small{Department of Mathematics, George Mason University, Fairfax, VA 22030}}
\affil[2]{\small{Department of Mathematics, University of North Carolina at Chapel Hill, Chapel Hill, NC 27599}}
\affil[1,4] {\small {Renaissance Computing Institute, University of North Carolina at Chapel Hill, Chapel Hill, NC 27517}}
\affil[3]{\small{Department of Mathematics, Brigham Young University, Provo, UT 84602}}
\affil[*]{\footnotesize{email: efleuran@unc.edu,
kslyman@live.unc.edu, blake@mathematics.byu.edu, ckrtj@renci.org

}\vspace{-4em}}
\date{}

\newcommand{\correction}[2]{#2}
\begin{document}

\maketitle

\bigskip

 \centerline{Dedicated to the memory of Hermann Flaschka}

\begin{abstract}
\noindent Analyzing when noisy trajectories, in the two dimensional plane, of a stochastic dynamical system exit the basin of attraction of a fixed point is specifically challenging when a periodic orbit forms the boundary of the basin of attraction. Our contention is that there is a distinguished {\em Most Probable Escape Path} (MPEP) crossing the periodic orbit which acts as a guide for noisy escaping paths in the case of small noise slightly away from the limit of vanishing noise. It is well known that, before exiting, noisy trajectories will tend to cycle around the periodic orbit as the noise vanishes, but we observe that the escaping paths are stubbornly resistant to cycling as soon as the noise becomes at all significant. Using a geometric dynamical systems approach, we isolate a subset of the unstable manifold of the fixed point in the Euler-Lagrange system, which we call the {\em River}.  Using the Maslov index we identify a subset of the River which is comprised of local minimizers.  The Onsager-Machlup (OM) functional, which is treated as a perturbation of the Friedlin-Wentzell functional, provides a selection mechanism to pick out a specific MPEP. Much of the paper is focused on the system obtained by reversing the van der Pol Equations in time (so-called IVDP). Through Monte-Carlo simulations, we show that the prediction provided by OM-selected MPEP matches closely the escape hatch chosen by noisy trajectories at a certain level of small noise.
\end{abstract}

\begin{flushleft}

\\
\end{flushleft}

\section{Introduction} \label{sec:intro}
It is well known that noise can work against the deterministic motion
of a dynamical system with an attracting fixed point. With probability
one, a noisy trajectory of a system with additive noise will, under
natural conditions, leave the basin of attraction of the fixed point,
assuming it is bounded, at some point in time. Large deviation theory
is devoted to finding the most probable escape path (MPEP) and the
expected time of escape. The MPEP can be thought of as the mode of
the probability distribution function of paths that escape from the
basin of attraction. The central results, which were largely formalized by
Friedlin and Wentzell \cite{freidlin_random_2012}, are asymptotic in the level of noise. 

\subsection{Stochastic Differential Equation}\label{intro_fwk}
Mathematically, the framework is an SDE of the form ($z\in\mathbb{R}^{n})$,
\begin{equation}
dz=F(z)dt+\sqrt{\varepsilon}\sigma dW\label{eq:SDE}
\end{equation}
This is a stochastic perturbation of the deterministic system given
by the drift term: $\dot{z}=F(z)$ where the noise strength is $\varepsilon$,
the structure of the noise is given by the $n\times n$ matrix $\sigma$,
and $W$ represents the standard 
\correction{comment 20}{Wiener}
process (here as a vector of separate processes). We will assume that
the drift vector field $F(z)$ is as smooth as needed, and generally
that $\sigma$ is the identity matrix. Freidlin-Wentzell theory has
its origins in the case of $F(z)$ being gradient: $F(z)=-\nabla V(z)$
for some potential $V$. Friedlin and Wentzell introduced the notion of a quasipotential
in order to generalize the theory and this provides a framework for
answering many questions. A particular challenge arises, however, when
the basin boundary is a periodic orbit. The work of Day \cite{day_exit_1996} gives
a clear and comprehensive picture of what happens in this case as
the noise vanishes. Day showed that there is no preferred exit point
or region along the periodic orbit and that the periodic motion causes
the most probable exit point to cycle around as the noise decreases.
The work of Maier and Stein \cite{maier_oscillatory_1996} also added to this picture and
a very detailed analysis was more recently given by Berglund and Gentz
\cite{berglund_noise-induced_2004,berglund_noise-induced_2014}

\subsection{Main Example}\label{intro_prob}
Monte-Carlo simulations of standard examples of Equation (\ref{eq:SDE}) where
$z\in\mathbb{R}^{2}$ and an attracting fixed point is surrounded
by an unstable periodic orbit, which is its basin boundary, show a
different picture in practice. The example that is invoked more than
any other of this scenario is the Inverted van der Pol equation (IVDP),
\correction{comment 15}{
\begin{equation}
\begin{aligned}
\dot{x}&=y, \\
\dot{y}&=-x+2\eta y(x^2-1).
\end{aligned}   
\label{eq:ivdp}
\end{equation}
which, when 
noise is added and
put in the form of Equation (\ref{eq:SDE}) in $\mathbb{R}^{2}$,
}
exhibits
a striking escape pattern that is at a rather definite part of the
periodic orbit. In Figure 1, we take $\eta=0.5$ and $\sqrt{\varepsilon}=0.3$2.
While this level of noise is decidedly not vanishingly small, it is
small and close to the limit of feasible Monte-Carlo simulations that
capture the MPEP without resorting to a strategy such as importance
sampling. The theory tells us that the escape will indeed be carried
around periodically. But it is also striking how stubbornly the escape
region shown in Figure 1 persists when the noise is small but not
vanishingly so. It can be seen in \cite{PhysRevLett.92.020601} how hard it is
actually to see the cycling, even to get one iteration of the escape
region half way around the periodic.

\begin{figure}[tbp]
\begin{center}
\includegraphics[height=6cm]{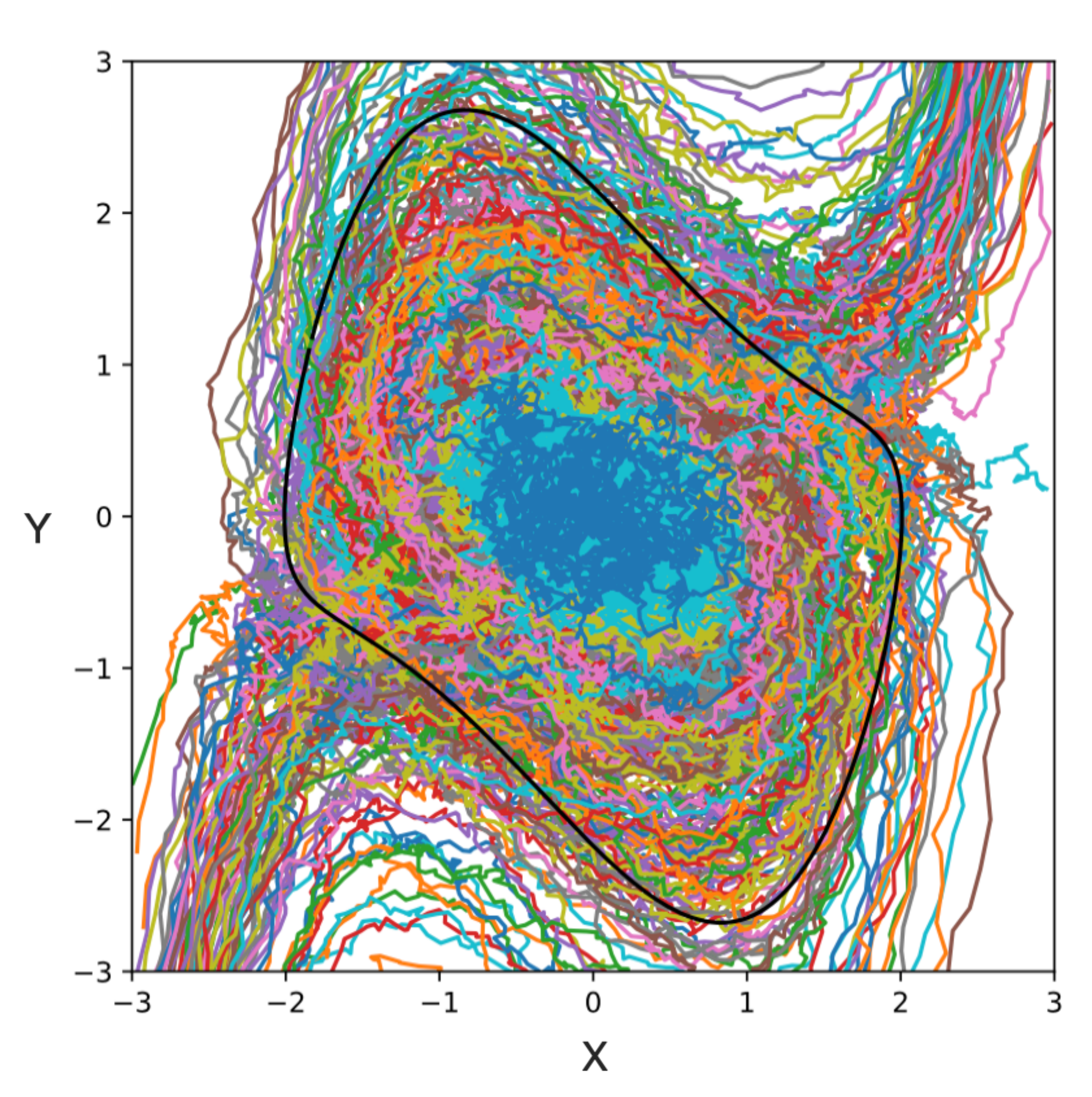}
\end{center}
\caption{Monte-Carlo simulations of noisy trajectories that escape the periodic orbit of the IVDP with noise strength $\sqrt{\varepsilon }=0.32$}
\label{fig:MPEPC1}
\end{figure}

\subsection{Dynamical Systems Approach}\label{intro_app}
We take the opposite viewpoint from the more theoretical literature
in this paper. Instead of trying to show that the theoretically predicted
cycling actually happens, we aim to show why non-cycling occurs at
small noise, but slightly away from the limit. The question we pose
is whether the evident escape region can be identified and understood
using the tools of the asymptotic theory. 

There are a number of different
approaches taken in analyzing the escape phenomenon, for instance:
matched asymptotics (WKB) \cite{PhysRevB.34.1572,PhysRevA.38.2605}, quasipotential theory (HJE)
\cite{day_exit_1996}, and a dynamical systems approach using random Poincaré maps
\cite{Niels1}. Various mixtures of these approaches have been
used, particularly in the physics literature \cite{PhysRevLett.92.020601,smelyanskiy_topological_1997,dykman_observable_1994,PhysRevE.72.036131}. 

In this work, a dynamical systems viewpoint is applied to the Euler-Lagrange
system derived for the Freidlin-Wentzell (FW) action functional of
the SDE with IVDP dynamics. We identify a specific part of the boundary
from a careful study of the geometric structure of the unstable manifold
of the fixed point in the Euler-Lagrange system. We call this set
of trajectories the \emph{River}. It is bordered by two specific
trajectories that are heteroclinic from the fixed point at the origin
to the periodic. It is shown by a number of authors \cite{PhysRevLett.92.020601, beri_solution_2005} that the global minimizer of the FW functional is a
hertoclinic and it acts as a
MPEP. 

Of particular interest are the
points on the periodic orbit where the river trajectories cross as they escape. We further find a subset of the river trajectories
that has a close correspondence with the escape hatch seen in the
Monte-Carlo simulations. To identify this set, we invoke the Onsager-Machlup
functional to account for the level of noise being small but away
from the asymptotic limit. 

There is some irony in our approach in that we are using the asympotic
theory (FW) to capture a phenomenon that we claim lies outside the
validity of that approach. One way to look at this is that we are
seeing which parts of the asymptotic theory extrapolate to this level
of noise and how it needs to be modified to capture the effects in
this parameter range. 

\subsection{Motivation}\label{sec:motiv}
Our motivation comes from thinking about physical systems relevant to the climate.  Tipping from an apparently stable state is particularly relevant in assessing climate change as abrupt changes are known to take place in critical climate subsystems. Moreover, the specter of their occurring in systems with large impact on climate functioning has made their study pressing, see \cite{Ritchie_Sieber,Wieczorek_Ashwin_Luke_Cox,Boulton2014,Lenton2011}. Three distinct types of tipping have been identified and distinguished from each other in the literature, see \cite{Ashwin_2017,dAPW}. One of these is through the response of a system to external noise. This noise may be identifiable as a known forcing of unknown magnitude, or an accounting of missing physics. Mathematically, noise is usually included through an additive stochastic perturbation of the underlying differential equations (DE). The complexity of the overall climate system makes this framework of stochastically perturbed DEs 
highly appropriate. Two questions then arise naturally: 
\begin{enumerate}
    \item Within a certain time horizon, what is the probability of tipping away from a given attracting state?
    \item What is the most likely trajectory that the system will take to tip? 
\end{enumerate}
 From the climate viewpoint, the first question addresses how dangerous the prospect of tipping might be, while the second will point to the way it will most likely happen. 

Tipping due to noise, so-called noise-induced or N-tipping, has largely been studied as a problem in Large Deviation Theory, see \cite{ forgoston_escape_2009, forgoston_primer_2018, tao_hyperbolic_2018, cameron_finding_2012, ryvkine_pathways_2006, grafke_numerical_2019}. As mentioned in Sec.\ \ref{intro_fwk}, large deviation theory comprises a body of results that are valid in the limit of vanishing noise. This is partly due to the connection of its development with molecular chemistry, but also that, from the mathematical viewpoint, it is a context in which rigorous analytical results can be obtained. As a consequence, the results obtained may only be valid for extremely long time scales, and these may be beyond what is relevant for the climate. In the climate context, we suggest that noise should be set to be small (otherwise the noise is trying to capture mechanisms that should be included in the model), but not vanishingly so. We shall refer to this as the intermediate noise case, although it is known as ``finite noise” in the physics literature, see \cite{Bandrivskyy_2003}.

\subsection{Plan of Paper}\label{intro_plan}
The paper is organized as follows. In Section \ref{sec:intro_key} we review some preliminary notions from Friedlin-Wentzell theory, discuss the key ideas and give a roadmap for the work. In Section \ref{sec:heteroclinics} we  discuss our methodology for computing unstable manifolds of fixed points and stable manifolds of periodic orbits, along with how we compute the heteroclinic orbits. Section \ref{sec:conjugatep} treats methods for computing conjugate points and how we relate our approach with the Maslov index. 
In Section \ref{sec:river} we discuss trajectories that exit the periodic orbit and relate them to (local) minimizers of a constrained variational problem. This is where we introduce the notion of the River and  of a pivot point. In Section \ref{sec:onsager} we describe our use of the Onsager-Machlup functional as a selection mechanism to pick out an MPEP for intermediate noise strength. In Section \ref{sec:mcsim}, we derive a converged distribution for the location of escape events. In Section \ref{sec:conclusion} we discuss how we match our simulations with the theory that we developed using the escape hatch, the pivot point, the OM point and the Monte Carlo simulations. Finally, Section \ref{sec:conclusion1} is devoted to discussing our approach and concluding remarks.
   
\section{Background and Key Ideas}\label{sec:intro_key}

We will work with a general set-up as given by Equation \eqref{eq:SDE} with $z\in\mathbb{R}^{2}$ and make assumptions that reflect the particular structure of interest. Some of these are very general, while some require explanation and verification in specific cases. Our viewpoint is that this latter type of assumptions would be verified numerically and we illustrate this for the case of IVDP, see Equation \eqref{eq:ivdp}. 
\subsection{Goal}\label{key_goal}
In the case of small but not vanishing noise, it is commented in Section \ref{intro_prob} that the ``escape hatch" through which noisy trajectories will favor exiting is fairly distinctive. Our goal is to show that this subset of the periodic can be clearly identified through a construction rooted in theory.

The construction will be based on finding structures in the dynamical Euler-Lagrange equations generated by finding extrema of the Freidlin-Wentzell action functional. 

\subsection{Set-up}\label{sec:key_setup}
Start with Equation \eqref{eq:SDE}, with $\sigma=I$, where $I$ is the $2 \times 2$ identity 

\begin{equation}\label{eq:1a_}
        dz = F(z)dt + \sqrt{\varepsilon}I dW,\\
\end{equation}

The first assumptions are on the underlying deterministic system
\begin{equation}\label{eq:1b_}
        \frac{dz}{dt} =  F(z),
\end{equation}
defined on \correction{comment 1}{$z = (x,y) \in \mathbb{R}^2$, $F=(f,g)$,} to capture the dynamics in which the escape of noisy paths, i.e. solutions of Equation \eqref{eq:1a_}, is through an unstable periodic orbit. 

\begin{description}
    \item[(A1)] There is an (exponentially) attracting fixed point of Equation \eqref{eq:1b_}, which we will assume is at $z=(0,0)$, and it is the only fixed point, 
    \item[(A2)] Equation \eqref{eq:1b_} has an (exponentially) repelling periodic orbit, which we denote by $\Gamma$, surrounding the attracting fixed point at the origin. Moreover there are no periodic orbits inside $\Gamma$.
\end{description}

The theory of large deviations tells us that the most probable paths of escape from the attracting fixed point through the unstable periodic orbit should minimize the Freidlin-Wentzell functional \cite{freidlin_random_2012}. In its most basic form, the functional is for paths $z=\varphi (t)$ on $[0,T]$
\begin{equation}\label{1c_}
    S_T(\varphi)=\frac12 \int_0^T \left|\dot{\varphi}-F(\varphi)\right|^2 dt,
\end{equation}
where $\dot{}=\frac{d}{dt}$. 

The most likely path from $z_0 \in \mathbb{R}^2$ to $z_1 \in \mathbb{R}^2$ is given by the path that minimizes Equation (\ref{1c_}) over absolutely continuous functions from $[0,T]$ to $\mathbb{R}^2$ with $z(0)=z_0$ and $z_1=z(T)$. The minimization procedure works well if, in reference to a system satisfying (A1) and (A2), the points $z_0$ and $z_1$ are inside $\Gamma$ and neither is the fixed point at the origin. If the paths involve $(0,0)$,  or cross $\Gamma$, then the situation is more complicated and this will be our focus. Nevertheless, the basic theory associated with the so-called action functional of Equation \eqref{1c_} underpins everything we will do. 

\subsection{Euler-Lagrange Equations}\label{key_EL}
The action functional of Equation (\ref{1c_}) can be written in terms of a Lagrangian 
\begin{equation}\label{1d_}
    S_T(\varphi)=\frac12 \int_0^T L(z,\dot{z}) dt,
\end{equation}
where, obviously $L(z,\dot{z})=\left|\dot{\varphi}-F(\varphi)\right|^2$. As in classical mechanics, the Euler-Lagrange equations for extrema of Equation (\ref{1d_})
can be written as a Hamiltonian system. 
We set

\[
p = \dot{x}-f,\quad q= \dot{y}-g.
\]
The Euler-Lagrange equations as a Hamiltonian system then reads,
\begin{equation}\label{eq:hamil}
\begin{split}
\dot{x}&= f+p\\
\dot{y}&= g+q\\
\dot{p}&= -f_xp-g_xq\\
\dot{q}&= -f_{y}p-g_{y}q.  
\end{split}
\end{equation}
The Hamiltonian is given by
\begin{equation}
    \label{hamil_func}
    H(x,y,p,q)=f(x,y)p+g(x,y)q+\frac{p^2+q^2}{2}
\end{equation}
A key point to note is that $p=q=0$ is invariant and that invariant plane carries the deterministic flow given by Equation (\ref{eq:1b_}), recalling that $z=(x,y)$. As a consequence, the fixed point 
\correction{comment 15}{at the origin} and periodic orbit $\Gamma$ reappear with their attraction and repulsion reproduced within the plane. Note that, with a slight abuse of notation we shall use the same notation of $O$ and $\Gamma$ for the fixed point and periodic orbit, respectively, in reference to both Equations (\ref{eq:1b_}) and (\ref{eq:hamil}).   Their stability properties change, however, in the full $4$-dimensional system of Equation (\ref{eq:hamil}).  This is the key to using Equation (\ref{eq:hamil}) for determining the most probable paths of escape from the attracting fixed point out of its domain of attraction. 

\correction{comment 3}{By assumption (A1), two of the eigenvalues of Equation (\ref{eq:hamil}) linearized at $O$ have positive real part,} which are the negative of the (deterministic) eigenvalues of Equation (\ref{eq:1b_}). Thus the unstable manifold of $O$, which we denote $W^u(O)$ is $2$-dimensional. \correction{comment 2}{The periodic orbit $\Gamma$ has one stable, two neutral, and one unstable Floquet multipliers. It may seem as though $\Gamma$ should have a 1-dimensional stable manifold, but integrating this 1-dimensional stable direction results in a tangent bundle, which will be a 2-dimensional manifold. It follows from Equation (\ref{eq:hamil}) that $\Gamma$ has a a 2-dimensional unstable manifold (in the deterministic plane) and a 2-dimensional stable manifold $W^s(\Gamma)$ which lies in the complement of the deterministic plane in $\R^4$.}

Both of these 2-dimensional objects will play central roles in this work, and the unstable manifold $W^u(O)$ will be the main focus. In the next section we relate it to the minimization procedure that renders the most probable paths. 

\subsection{The Quasipotential and \texorpdfstring{$W^u(O)$}{WuO}}\label{key_unstmfld}
Of particular interest are paths starting at the fixed point and escaping its basin of attraction, i.e., getting outside $\Gamma$. First, consider paths that reach some point $z^*$ possibly inside $\Gamma$ from a start at the fixed point. The formulation of the action functional suggests that we seek paths going from  the fixed point to $z^*$ in time $T$. The time it takes to reach $z^*$ is something we want to keep free, however, and so the following quantity, called the {\em Quasipotential}, see \cite{freidlin_random_2012}, is defined as

\begin{equation}
    \label{quasipot}
    V(z^*)=\inf_{T>0, \varphi \in X_T} S_T (\varphi ),
\end{equation}
where $X_T$ is the set of absolutely continuous functions satisfying the boundary conditions: $\varphi (0)=0$, i.e., the fixed point, and $\varphi (T) =z^*$. It follows from Lemma 3.1 in \cite{freidlin_random_2012} that any minimizer realizing the infimum in Equation \eqref{quasipot} must lie in the set $H=0$. The only point in $H=0$ with $z=(0,0)$ is the fixed point of Equation \eqref{eq:hamil}, i.e., with $p$ and $q$ also equal to $0$. 

Since the only access to the fixed point at $O$ in the zero-set of the Hamiltonian of the 4D system of Equation (\ref{eq:hamil}) is on the unstable manifold $W^u(O)$, it follows that any minimizer must lie in $W^u(O)$ and the domain on which any minimizer is defined must be semi-infinite. By a reparameterization, if necessary, it can be taken to be $(-\infty, 0]$. 

For $z^*$ inside $\Gamma$, the infimum in Equation (\ref{quasipot}) is realized by a trajectory on $W^u(O)$. An important point is that for $z^*\in \Gamma$, this is not the case even though there may be trajectories on $W^u(O)$ that cross $\Gamma$. 
 \begin{remark}

By a quirk of the way the quasipotential is defined, the minimizer is not actually in the space on which the functional is defined, namely $X_T$, since a minimzing path cannot reach $O$ in finite (backward) time. This is rectified by considering what is called the Geometric Minimum Action, see \cite{heymann_geometric_2008}. The geometric action has the effect of reparametrizing the paths so that they all lie on a fixed bounded domain. One way this is achieved is to use arc length to parameterize the paths. Since the paths on $W^u(O)$ with fixed end point $z$ inside $\Gamma$ have finite arc length, the minimizing path does lie in the set of paths over which the geometric action is minimized. Note that this does not work for the trajectories in $W^s(\Gamma)$ as the arc length of any trajectory tending to $\Gamma$ is necessarily infinite. It is for this reason that we do not directly use the geometric minimum action in this work.

\end{remark}

\subsection{Singularities of the Quasipotential and Folding of \texorpdfstring{$W^u$}{Wu}}

It is well known that the quasipotential is not in general smooth. Caustics can form, see \cite{smelyanskiy_topological_1997}, and there might be multiple minimizers of the action functional with the same $z^*=(x,y)$ value. Viscosity solutions of the associated Hamilton-Jacobi equation are invoked to sort which is the global minimizer (infimum) for that value of $z^*$, see \cite{cameron_finding_2012}. 

The signature of a singularity of the quasipotential in the unstable manifold is a fold in the manifold when projected onto $(x,y)$-space. Indeed, if over a set $U \subset \R^2$, the unstable manifold $W^u(O)$ is given by the graph of a function $(p,q)=h(x,y)$ for $(x,y)\in U$, then the quasipotential will be smooth on $U$. Folds can be detected by finding conjugate points (Definition \ref{def:conj_pt}) along trajectories, see Section \ref{sec:conjugatep}. 
In actual fact, it is unlikely that the full unstable manifold is the graph of a function of $(x,y)$. This is because of the complex tangling that occurs when there are transverse intersections of stable and unstable manifolds along heteroclinic orbits. By identifying where these folds happen and looking at trajectories on $W^u(O)$ up to these fold points, we can obtain a clear picture of the quasipotential in large regions inside $\Gamma$. The key is the fold points are related to the minimzation of the action functional as they are conjugate points as used in the calculus of variations. To see this, note that at a fold point the tangent space to $W^u(O)$ will have a vertical tangent vector. This forces there to be a conjugate point. 

Since we are interested in trajectories on $W^u(O)$, we make a definition of conjugate point that is tailored to this situation.

\begin{figure}[tbp]
\begin{center}
$
\begin{array}{c}
\includegraphics[scale=0.5]{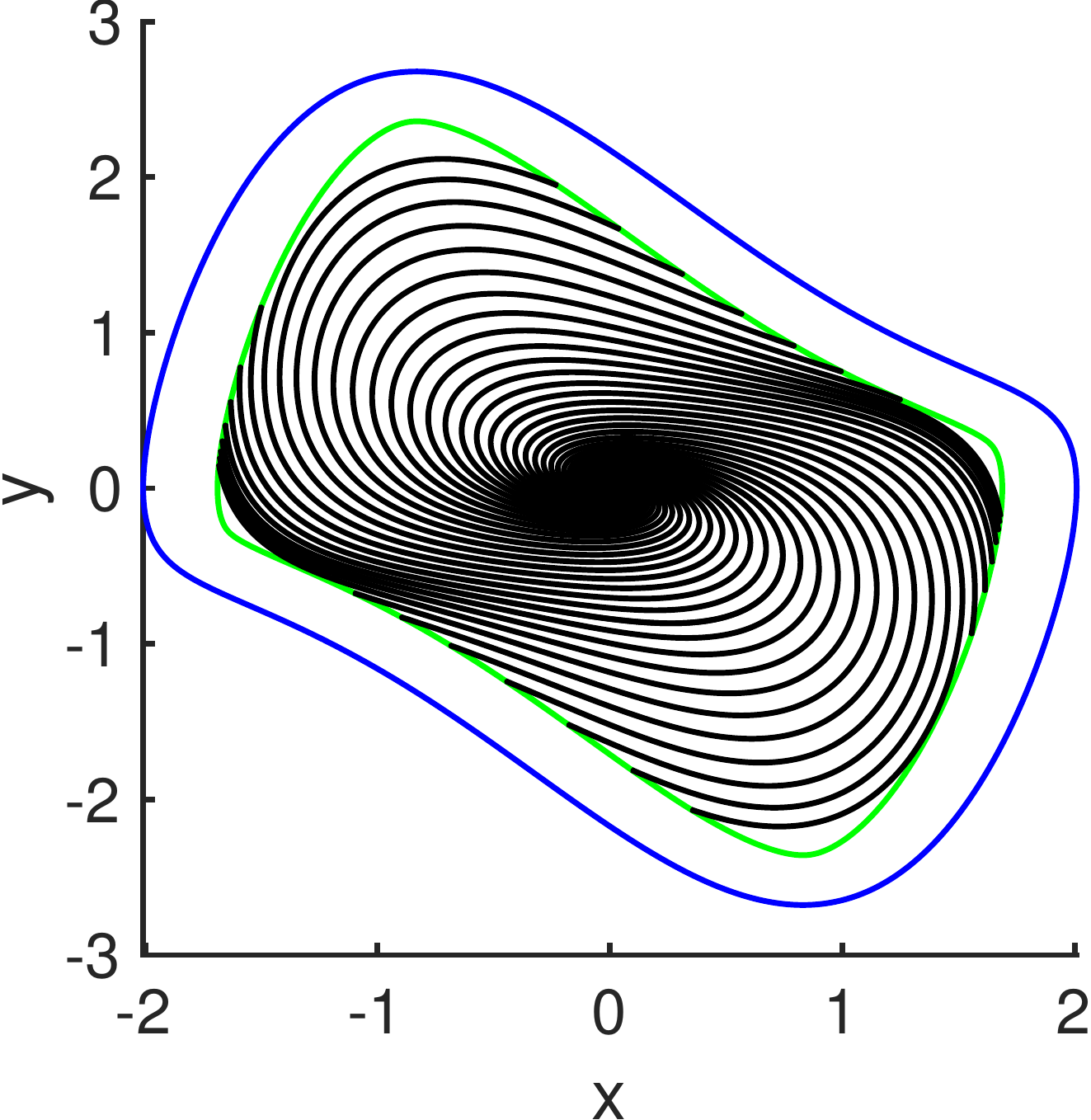} 
\end{array}
$
\end{center}
\caption{Plot of paths sampled from the entire unstable manifold of $O$, indicating no conjugate points occur before a path reaches the 0.32 border (in green).}
\label{fig:conj_no_conj}
\end{figure}

\begin{definition} \label{def:conj_pt}
If $z(t)$ is a trajectory of Equation (\ref{eq:hamil}) on $W^u(O)$ then $\tau$ is said to be a conjugate point if the projection of the tangent space $T_{z(\tau)}W^u(O)$ to $(x,y)$ space is not of full rank (i.e., not onto). 
\end{definition}

This corresponds to the classical definition of a conjugate point extended to the case of an extremizing trajectory on a semi-infinite domain. \correction{comment 4}{Indeed, if $\tau$ is a conjugate point then there will be a} solution of the linearization of Equation (\ref{eq:hamil}), denoted $U(t)$, along $z(t)$ that satisfies $u(\tau)=0$ and $u(t) \rightarrow 0$ as $t \rightarrow -\infty$ where $U=(u,w)$ and $u$ is the two-dimensional (linearized) variable corresponding to $(x,y)$. 

The construction of $W^u(O)$ is achieved by taking a small circle around $O$ inside $(x,y)$ space and growing it under the flow. This is explained in Section \ref{sec:heteroclinics}. We can see how far the unstable manifold can be grown without hitting a conjugate point along any of the trajectories. 

We set a $\delta-$collar of the periodic orbit as the set of points (in $(x,y)$ space) inside $\Gamma$ that are within a distance $\delta$ of $\Gamma$. It is known that, see \cite{berglund_noise-induced_2004, berglund_noise-induced_2014}, if the trajectories are within \correction{comment 12}{$\mathcal{O}(\sqrt{\epsilon})$} of the periodic then the diffusion will dominate and cycling will not play a significant role. It is thus interesting to see if we can reach the \correction{comment 12}{$\delta \sim \mathcal{O}(\sqrt{\epsilon})$} collar without hitting a conjugate point. With $\sqrt{\epsilon}=0.32$ in the IVDP system, we see that there are, in fact, no conjugate points between this collar and the fixed point, see Figure \ref{fig:conj_no_conj}. This has the consequence that the quasipotential is smooth in this set as there will be no folding until the collar is reached, where, as stated above, diffusion takes over.

\subsection{Heteroclinic Orbits}\label{sec:hetero1}
Evaluating the quasipotential for $z^* \in \Gamma$ requires special consideration. It must be constant on $\Gamma$ as it costs nothing in terms of $V$ to traverse the periodic orbit since it is an orbit of the deterministic system. From general variational arguments, there must be a minimizing trajectory, although its domain may not be finite, or even semi-infinite. Since the minimizing trajectory must be smooth, it follows that it must be a heteroclinic orbit from $O$ to $\Gamma$. In particular, its domain will be $(-\infty, +\infty)$ and it lies in both $W^u(O)$ and $W^s(\Gamma)$.
\begin{figure}[tbp]
\begin{center}
\includegraphics[height=4cm]{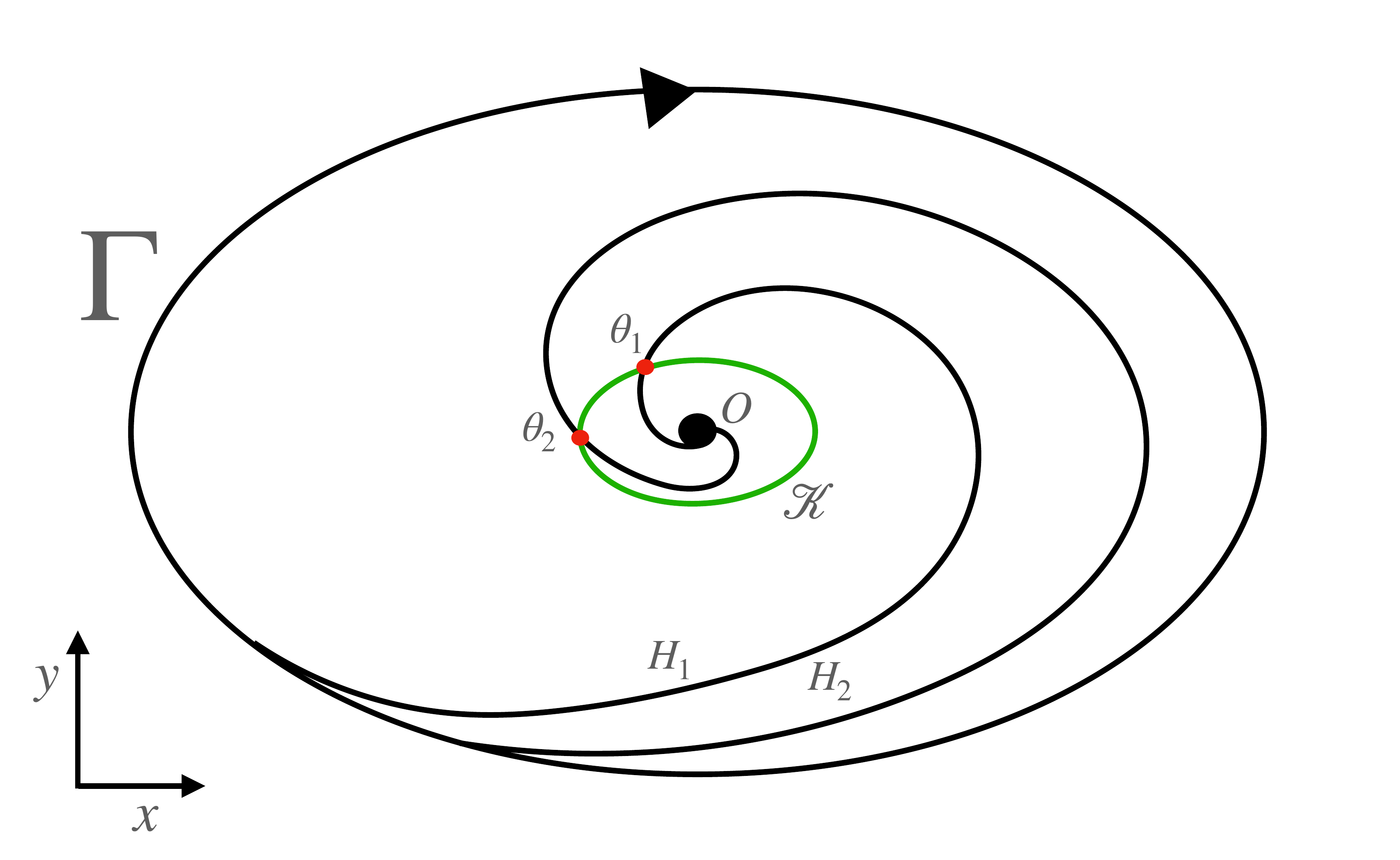}
\end{center}
\caption{Orbits on $W^u(O)$ can be parameterized by a simple curve $\mathcal{K} \subset W^u$ that surrounds $O$. $\mathcal{H}_1$ and $\mathcal{H}_2$ can be obtained from $\theta_1, \theta_2 \in [0, 2\pi)$.}
\label{fig:Hetero_River1}
\end{figure}

Generically, we expect there to be more than one heteroclinic orbit. Indeed if the minimizing heteroclinic is realized by a transverse intersection of $W^u(O)$ and $W^s(\Gamma)$, then there must be, at least, a second one. In the IVDP equation, we find exactly four, but there is a symmetry in that problem and so there are only two independent heteroclinics, the other two being given as a mirror image. We can distinguish trajectories that are minimizers by their Maslov Index \correction{comment 19}{\cite{doi:10.1098/rsta.2017.0187, Arnold}}. In the following definition, $Z=\{z(t)|t \in I\}$ will be a trajectory of Equation (\ref{eq:hamil}) on some interval $I\subset \R$, which may be bounded, infinite or semi-infinite. 

\begin{definition}
    The {\em Maslov Index} of a trajectory $Z$  on $W^u(O)$, denoted $m(Z)$, is the number of conjugate points in $t \in I$ along $z(t)$, counting multiplicity. 
\end{definition}

\noindent The multiplicity is that of the projection in the definition of conjugate point.

We make the following assumptions about the configuration of heteroclinic orbits in Equation (\ref{eq:hamil}) which we will refer back to in Section \ref{sec:conjugatep} and Section \ref{sec:river}, and verify for IVDP.
\begin{description} 
 \item[(A3)] There are two heteroclinic orbits of Equation \eqref{eq:hamil}, denoted $\mathcal{H}_1$ and $\mathcal{H}_2$, both from $(0,0)$ at $-\infty$ to $\Gamma$ at $+\infty$. Moreover $W^u(O)$ and $W^s(\Gamma)$ intersect transversely along each of them.
 \end{description}
 We note that $W^u(O)$ and $W^s(\Gamma)$ live in $H = 0$. If $x$ and $y$ are fixed, then a convenient description can be given for $H=0$. By completing the square inside the expression for $H=0$ the coordinates $(p,q)$ satisfy 
 
\begin{equation}\label{eq:Ham_torus}
(p+f(x,y))^2+(q+g(x,y))^2= f^2(x,y)+g^2(x,y),
\end{equation}
which is a circle with center at $(f(x,y), g(x,y))$ and radius $(f^2(x,y)+g^2(x,y))^\frac12$.
    
Let $C\subset \R^2$ be a simple closed curve, parameterized by $s\in[0,1]$, enclosing the fixed point $O$ and lying inside the periodic orbit $\Gamma$. Putting the circles together for each $s\in[0,1]$ yields a torus $\mathcal{T}_C \subset \R^4$. The intersection of $W^u(O)$ with $\mathcal{T}_C$, and of $W^s(\Gamma)$ with $\mathcal{T}_C$, generically are closed curves in the toroidal direction of $\mathcal{T}$ that do not wrap in the \correction{comment 20}{poloidal} direction. Indeed, the intersection cannot wrap in the \correction{comment 20}{poloidal} direction because this would correspond to $W^u(O)$ intersecting the $p = q = 0$ plane, which it cannot do since that plane is invariant. Thus, if these two closed curves intersect once, they must intersect at least twice. That is, there exists a second heteroclinic connection $\mathcal{H}_2$. In general, and generically, the number of crossings must be even, corresponding to an even number of heteroclinic connections. This gives a justification for seeking two heteroclinic orbits. We will distinguish them through their Maslov indices.

\begin{description}
  \item[(A4)] The Maslov index of $\mathcal{H}_1$ is $0$ and of $\mathcal{H}_2$ is $1$.  
\end{description}

Since an orbit with no critical points will minimize the action functional, the heteroclinic $\mathcal{H}_1$ will be a minimizer. On the other hand, we know that $\mathcal{H}_2$ will not be. A priori, we do not know that $\mathcal{H}_1$ is a global minimizer as there may be others with Maslov Index equalling 0. Generically, there will be a finite number and the global minimizer is found just by conducting a search through action values. In the IVDP system, the 0-Maslov Index trajectory (and its mirror image) are verified to be minimizers as there are no others.

\subsection{Exit Trajectories}\label{subsec:ke y_exit}
A characteristic of 2-dimensional systems such as the one we are considering is that, at least for generic problems, some trajectories on $W^u(O)$ exit the periodic orbit (when projected onto the $(x,y)$ plane.) This cannot happen in 1-dimensional, nor in gradient systems in 2D (which are not generic). Moreover, we can show that some of these exit trajectories are (local) minimizers of the FW action functional, see Section \ref{sec:river}. 

There must be at least one heteroclinic connection between the origin $O$ and the periodic $\Gamma$. Indeed, there exists a most probable escape path that is a heteroclinic connection. 
We are assuming in (A3) that there are at least two heteroclinic connections between $O$ and $\Gamma$. For IVDP, we numerically verify the existence of four such heteroclinic connections. We expect that \correction{comment 20}{heteroclinic} connections come in pairs. 

When $W^u(O)$ and $W^s(\Gamma)$ intersect transversely along the heteroclinic orbits then complex tangling will occur. This is a familiar picture in dynamical systems that is related to a homoclinic tangle and the Shilnikov mechanism for chaotic dynamics. While this picture is very complicated, if we view it in terms of finite portions of $W^u(O)$ as it is built up, then some clarity over the main trajectories that play a role in guiding the stochastic trajectories can be obtained. In Section \ref{sec:heteroclinics}, a procedure is articulated for growing the unstable manifold. By a finite portion of $W^u(O)$ we mean the unstable manifold grown out to a fixed finite time from its generating circle.  The following lemma spells out that there must exist exit trajectories if there is a transverse heteroclinic.

\begin{lemma} \label{lemma:exit}
In the neighborhood of a transverse heteroclinic orbit  with Maslov Index 0 (as in (A3)), there are trajectories on $W^u(O)$ which exit $\Gamma$ and, moreover have no conjugate points before exiting $\Gamma$.
\end{lemma}
 The proof of this lemma is a standard dynamical systems proof based on the observation that, inside the 3D set $H=0$, $W^u(O)$ will straddle $W^s(\Gamma)$, by transversality, and one part will have to exit $\Gamma$. The fact that there will be no conjugate points of the exiting trajectories before exit follows from continuity of the tangent space to $W^u(O)$ as trajectories are perturbed. 

 Over the periodic orbit, the zero energy level $H=0$ is a torus as indicated above since $\Gamma$ is a simple closed curve itself. Consistently with the above notation. this torus is denoted $\mathcal{T}_{\Gamma}$ and any trajectory exiting the periodic orbit (when projected onto $(x,y)$-space) must exit through $\mathcal{T}_{\Gamma}$. 

 A key object for understanding the MPEP structure for intermediate noise is the set $W^u(O) \cap \mathcal{T}_{\Gamma}$. Due to the tangling of $W^u(O)$ this set will be very complicated. But we will isolate a subset of it, using the Maslov Index, that we argue gives considerable insight into the escape hatch noted from Monte-Carlo simulations. This will be the subject of Section \ref{sec:river}. In order to describe this set properly, we first need to delve further into the way we compute the various invariant manifolds.

\section{Computing Stable and Unstable Manifolds
\label{sec:heteroclinics}}

Computing the unstable manifold proceeds in two steps. For the first part we invoke a highly accurate method for calculating the local unstable manifold near the fixed point at $O$. Since the manifold is 2-dimensional, the full manifold can be generated by initiating trajectories from a circle inside the local unstable manifold. The set of trajectories so constructed form the global, or full, unstable manifold. 

Except for the heteroclinic orbits themselves, we are interested in trajectories that reach the periodic orbit $\Gamma$ in finite time. Therefore we can focus on a finite portion of $W^u(O)$. Such a finite portion can be generated up to any desired accuracy by going to high enough order in the method described next.

\subsection{Computing  the Local \texorpdfstring{$W^u(O)$}{WuO}}
We use the parameterization method of \cite{Cabre1, Cabre2, Cabre3, Haro2, MR2276478, FLEURANTIN2020105226} to accurately compute the local unstable manifold of the fixed point $O$, and grow the unstable manifold in order to compute the heteroclinic connections and the set we call the River.  By doing so, we obtain a high-order approximation of $W^u(O)$. The parameterization method lays out a general functional analytic framework for studying invariant manifolds in a number of different contexts and applications. The method is constructive and leads to efficient and accurate numerics. The main idea is to examine an invariance equation describing the invariant manifold. One plugs in a certain formal series into the invariance equation and solves the problem via a power matching scheme.

Given an analytic vector field $F: \R^4 \to \R^4$ with $F(0)=0$, and the conditions for a 2-dimensional unstable manifold, the parameterization method seeks an embedding $P: B_1^2(0) \to R^4$, with $P(0)=0$, $B_1^2(0)$ the unit disk centered at $0$, and a linear vector field $R: B_1^2(0)\to \R^2$ such that
\begin{equation}\label{eq:invariance}
F \circ P(x) = DP(x) Rx.
\end{equation}
In other words, the goal is to obtain a conjugacy between the flow on an invariant manifold of interest and the associated linear problem restricted to the unit disk. \correction{comment 16}{We always restrict $P$ to the unit disk for the sake of numerical stability.}

In our case, the spectrum of $DF(0)$ is composed of distinct eigenvalues $\{ \lambda_i^u \}_{i=1}^2$, $\{ \lambda_i^s \}_{i=1}^2$, where the real parts are positive for $\lambda_i^u $ and negative for $\lambda_i^s$ (the superscripts $s$ and $u$ stand for stable and unstable respectively). We can then take $R$ to be the usual (real) matrix associated with two complex conjugate eigenvalues.

 The image of $P$ is a smooth 2-dimensional manifold, and since it will be invariant by Equation \eqref{eq:invariance}, it is an invariant manifold for $0$ in $\R^4$. Furthermore, if we denote by $\Phi:\R^4 \times \R \to \R^4$ the flow generated by $F$ and note that from \cite{haro_canadell_2018}, $P$ must satisfy Equation \eqref{eq:invariance} if and only if \[\Phi(P(x),t)=P(e^{Rt}(x)),\]
 for all $x$ and $t$ for which it is defined, and thus it is a local unstable manifold for the vector field $F$ at $0$.

The global unstable manifold can then be obtained by integrating forward trajectories from this local unstable manifold. In the next subsection, we describe a systematic way of doing this which will allow to give a convenient description of the set of exit trajectories we claim to be significant in understanding the escape hatch of the stochastic system.

\subsection{Generating the Full \texorpdfstring{$W^u(O)$}{WuO}}
The circle $\mathcal{K}=P(\partial B^2_1(0))$ inside the local unstable manifold will be used to generate the full $W^u(O)$. An explicit parameterization of  $\mathcal{K}$ will be given that also gives more insight into how the parameterization method works for approximating the local unstable manifold. 

The first step if to extend the real-analytic vector field $F$ on $\R^4$ to a complex analytic vector field on $\mathbb{C}^4$. A parameterization $\hat{P}(z_1,z_2)$ of $W^u(O)$ in $\mathbb{C}^4$ is then sought. The map $\hat{P}$ is taken to be a double infinite sum $$\hat{P}(z_1,z_2) = \sum_{m=0}^{\infty}\sum_{n=0}^{\infty}\alpha_{mn}z_1^mz_2^n.$$ To relate this to the flow, we can think of $z_1 = z_1^0e^{\mu_1 t}$, $z_2 = z_2^0e^{\mu_2 t}$, $z_1^0,z_2^0\in \mathbb{C}$, and $\mu_1$ and $\mu_2$ are the (complex) unstable eigenvalues of the Jacobian of Equation \eqref{eq:hamil} evaluated at the fixed point $O$. Since $\mu_1$ and $\mu_2$ are complex conjugate pairs, as are their eigenvectors, the value of $\hat{P}(z_1,z_2)$ is real if $z_1$ and $z_2$ are complex conjugate pairs. Another way of saying this is that when $\mu_1$, $\mu_2$ are complex conjugates, the coefficients of $\hat{P}$ have the symmetry $\bar\alpha_{nm}=\alpha_{mn}$ for all $m+n \geq 2$. Choosing complex conjugate eigenvectors $\xi_1$ and $\xi_2$ and setting $\alpha_{00}=O$, $\alpha_{01}=\xi_1$, $\alpha_{10}=\xi_2$ enforces the symmetry to all orders. The power series solution of $\hat{P}$ has complex coefficients, but we get the real image of $\hat{P}$ by taking complex conjugate variables. That is, we define, for example, for the real parameters $\zeta_1, \zeta_2$, the function:
\[P(\zeta_1, \zeta_2)=\hat{P}(\zeta_1 + i \zeta_2, \zeta_1 - i \zeta_2),\]
which parameterizes the real unstable manifold. Further, we scale the eigenvectors we use in the construction of $P$, which affect $z_1^0,z_2^0$, so that the double infinite sum converges whenever $|z_1|,|z_2|\leq 1$.

We parameterize $\mathcal{K}$ with $\theta \in[0,2\pi)$ by setting $z_1(\theta) = \cos(\theta)+i\sin(\theta)$ and $z_2(\theta) = \cos(\theta)-i\sin(\theta))$. 
We then define
\begin{equation}
    \mathcal{K}:= \{\hat{P}(z_1(\theta),z_2(\theta)): \theta \in [0,2\pi)\}.
\end{equation}\label{eq:angle}
That is, $\mathcal{K}$ is a simple closed curve in $W^u(O)$, parameterized by $\theta\in[0,2\pi)$, whose projection onto the $(x,y)$-plane is an ellipse enclosing $(0,0)$. We denote by $\mathcal{K}(\theta)$ the point of $\mathcal{K}$ corresponding to $\theta$. The curve $\mathcal{K}$ is depicted as a green circle in Figure \ref{fig:Hetero_River1}. 

\subsection{Computing \texorpdfstring{$W^s(\Gamma)$}{WsG}}
Next, in order to compute the stable manifold of the periodic orbit, we proceed as follows. We first divide  up the periodic orbit into $N$ points $\Gamma_k$, $0 \leq k \leq N$. To each point there corresponds a time $\tau_k$ such that \correction{comment 5}{$\Gamma_k = \Gamma(\tau_k)$,} where here $\Gamma(t)$ is the periodic orbit.
\begin{figure}[tbp]
\begin{center}
\includegraphics[height=6cm]{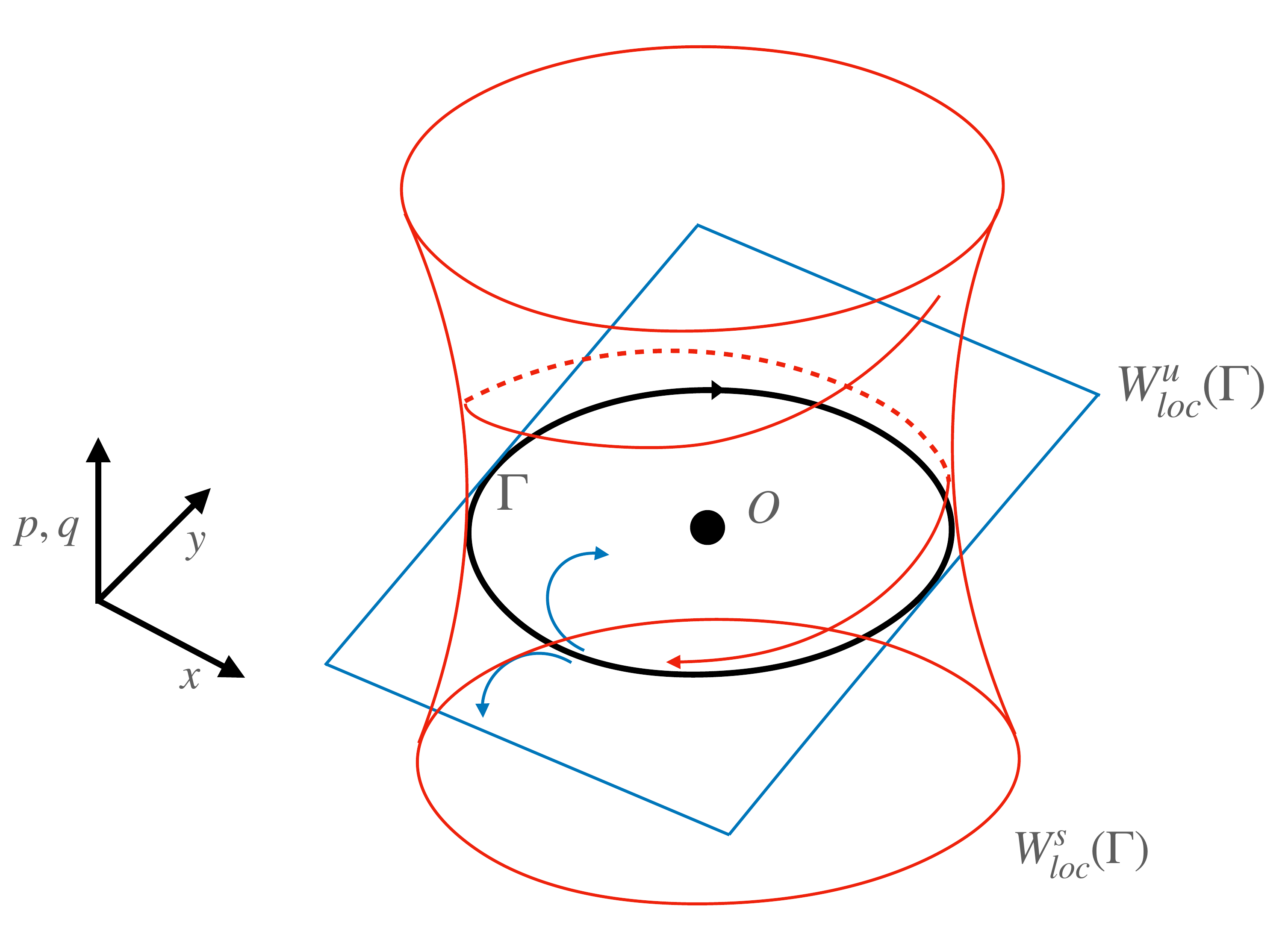}
\end{center}
\caption{Schematic of the flow near the periodic orbit $\Gamma$.}
\label{fig:Periodic-Schem}
\end{figure}

Let  $\xi_{s_0}$ be the eigenvector associated with the stable eigenvalue of the monodromy matrix. This vector is tangent to the stable manifold at $\Gamma_0$. The stability type of the state transition matrix of any point $\Gamma_k$ on the periodic orbit is independent of $k$ and the eigenvectors can be computed if the state transition matrix is known at a base point $\Gamma_k$. They are just the eigenvectors of the monodromy matrix (which is computed at $\Gamma_0$) multiplied by the state transition matrix of the new point $\Gamma_k$,
    \[
    \xi_{s_k} = \phi(0, \tau_k)\xi_{s_0}.
    \]
Then for each $\Gamma_k$ we compute the state transition matrix $\phi(0, \tau_k)$ and from this obtain the tangent space to the stable manifold there, by the formula above. Set a tolerance 
\correction{comment 17}{$\nu$}. If $\nu$ is small enough then 
    \[
    x_{s_k}(0) = \Gamma_k \pm \nu \xi_{s_k}
    \]
    are points very nearly on the stable manifold. We integrate these initial conditions over some time interval $[0, T_f ]$ obtaining the orbits $x_{s_k}(t)$. Then $t$ along these orbits is the second coordinate on the manifold.  By varying $k$ and $t$ we obtain a good approximation of the stable manifold. A schematic of the flow near $\Gamma$ can be seen in Figure \ref{fig:Periodic-Schem}.

\subsection{Obtaining the Heteroclinic Orbits}

Finally, we compute the heteroclinic orbits resulting from the transverse intersections of $W^u(O)$ and $W^s(\Gamma)$. 
This is achieved in two steps. We first compute the heteroclinic orbits by looking at the transverse intersections of the invariant manifolds and using an algorithm to find the two closest points (one from a trajectory from the stable manifold of the periodic orbit, and another from a trajectory from the unstable manifold of the fixed point at the origin). So we initially compute the heteroclinics using those two points integrating forward and backward in time. We then use that trajectory to find its corresponding angle on the parameterized circle $\mathcal{K}$ (as an initial guess) which we then refine to compute the heteroclinic orbit but now by integrating a single point, forward and backward in time. 

The result is illustrated for IVDP in Figure \ref{fig:stabunstab}. In the first two parts (a) and (b), we compute  $W^u(O)$ for Equation \eqref{eq:hamil} and $W^s(\Gamma)$  respectively. The parameters are set as $\sqrt{\varepsilon}=0.3$, and $\eta=0.5$. In Figure \ref{fig:stabunstab} (c), we delineate the transverse intersections of $W^u(O)$ and $W^s(\Gamma)$ in green and black in $(x,y,p)$ space. In Figure \ref{fig:stabunstab} (d), we remove most of $W^u(O)$ and $W^s(\Gamma)$ leaving only the heteroclinic orbits (green and black curves), and can clearly see that the intersections of the two manifolds occur along 4 distinct curves.

\begin{figure}[tbp]
\begin{center}
$
\begin{array}{cc}
\includegraphics[scale=0.4]{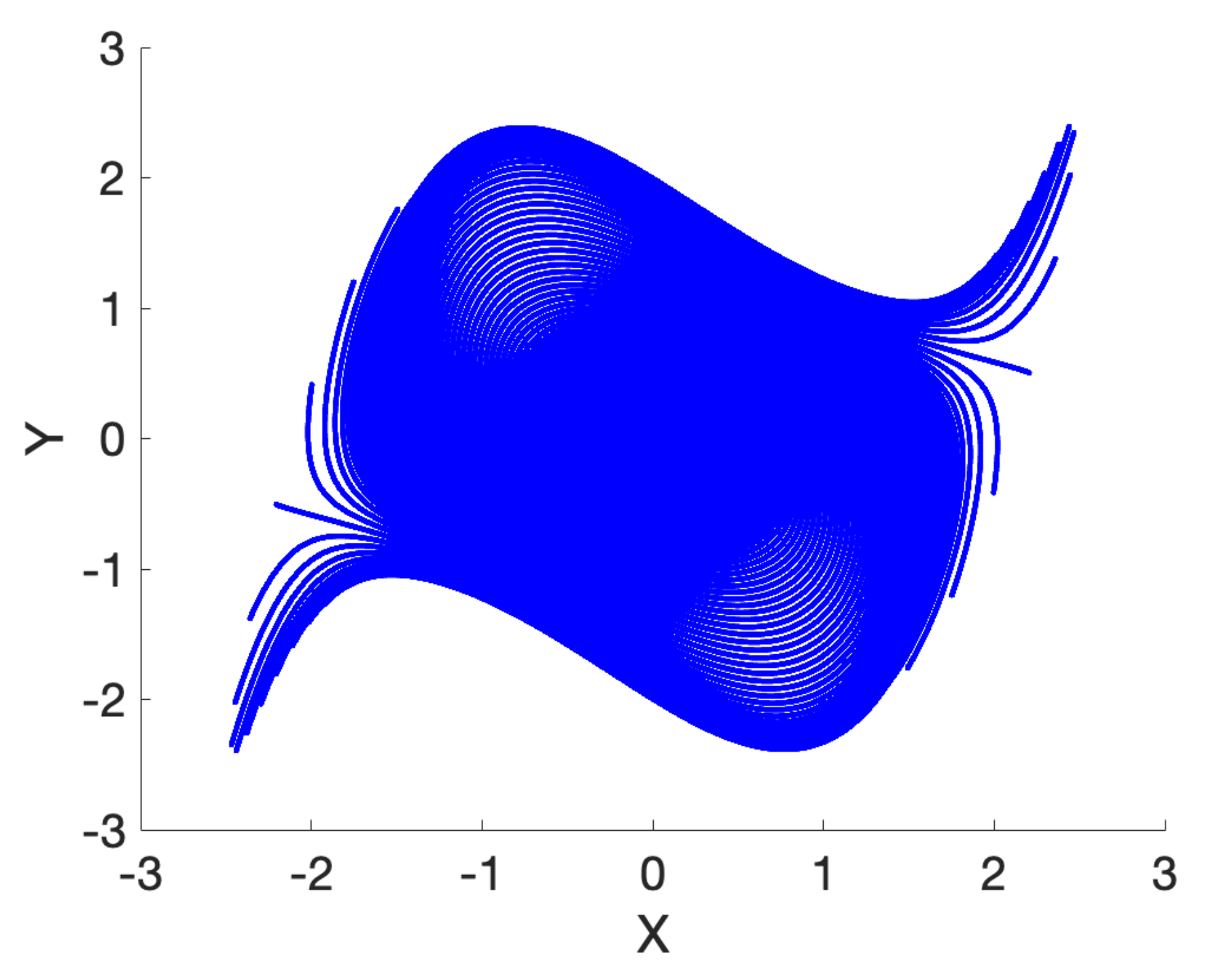} &
 \includegraphics[scale=0.4]{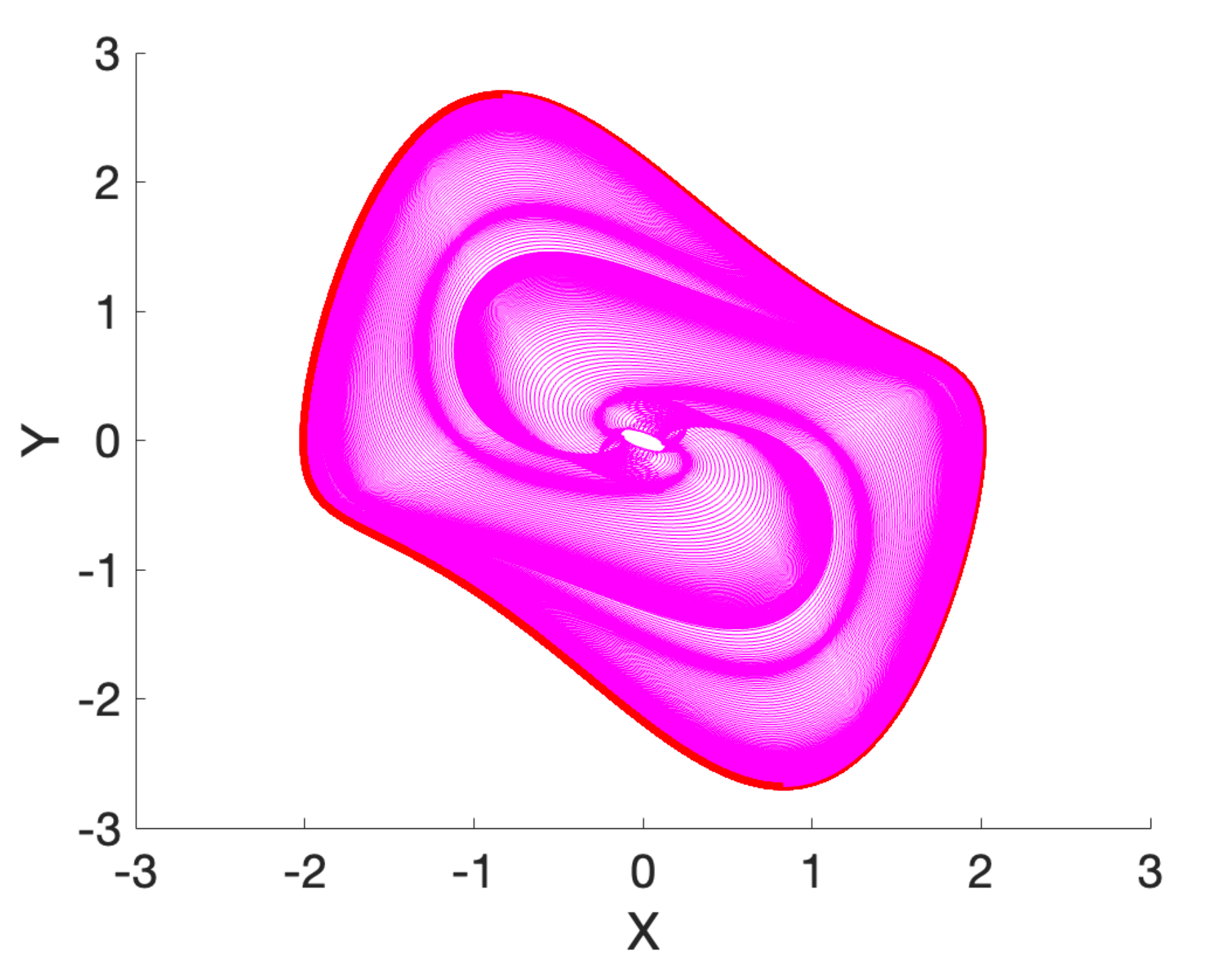}\\
 (a) & (b)\\
\includegraphics[scale=0.4]{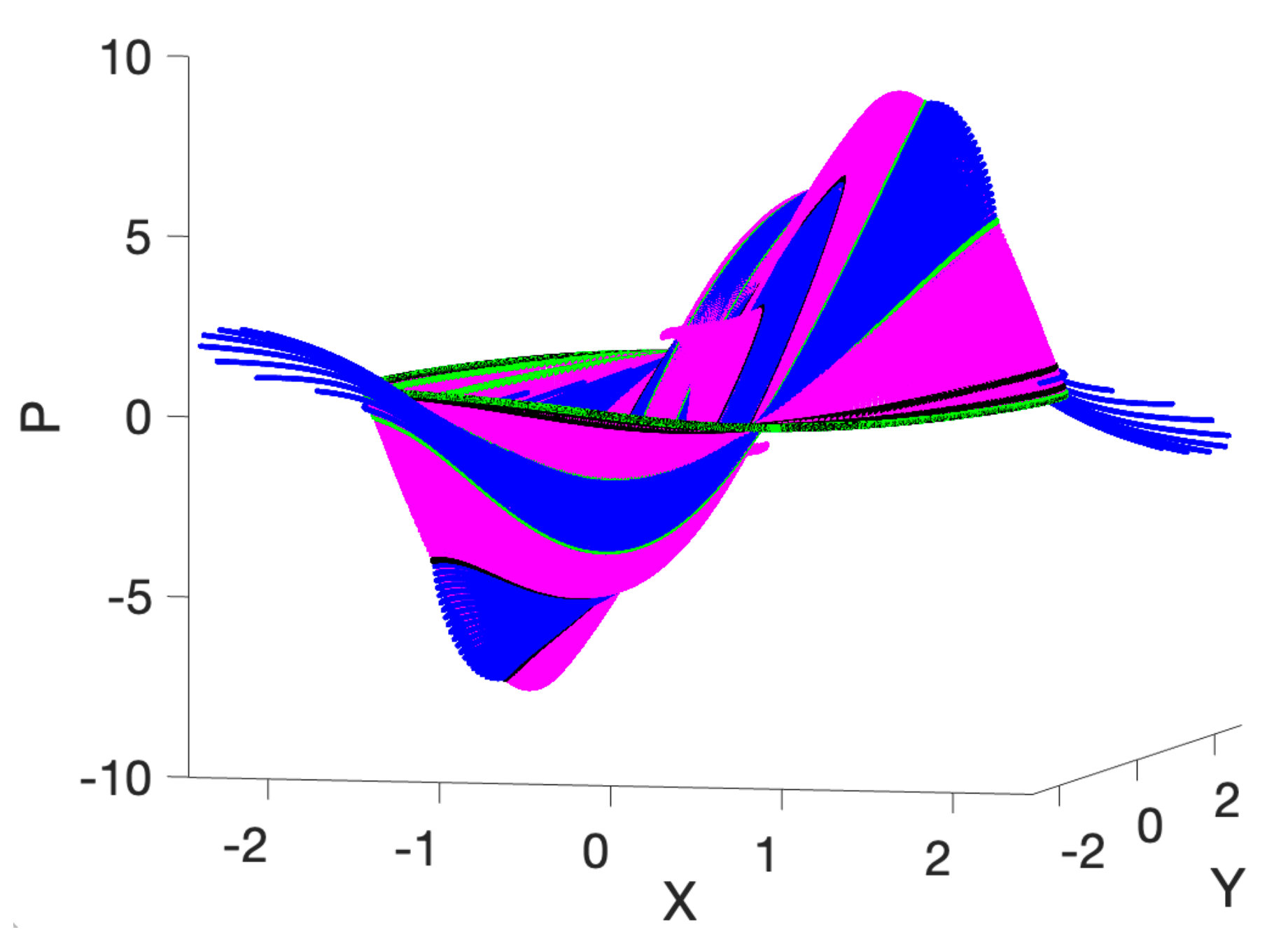} &
 \includegraphics[scale=0.4]{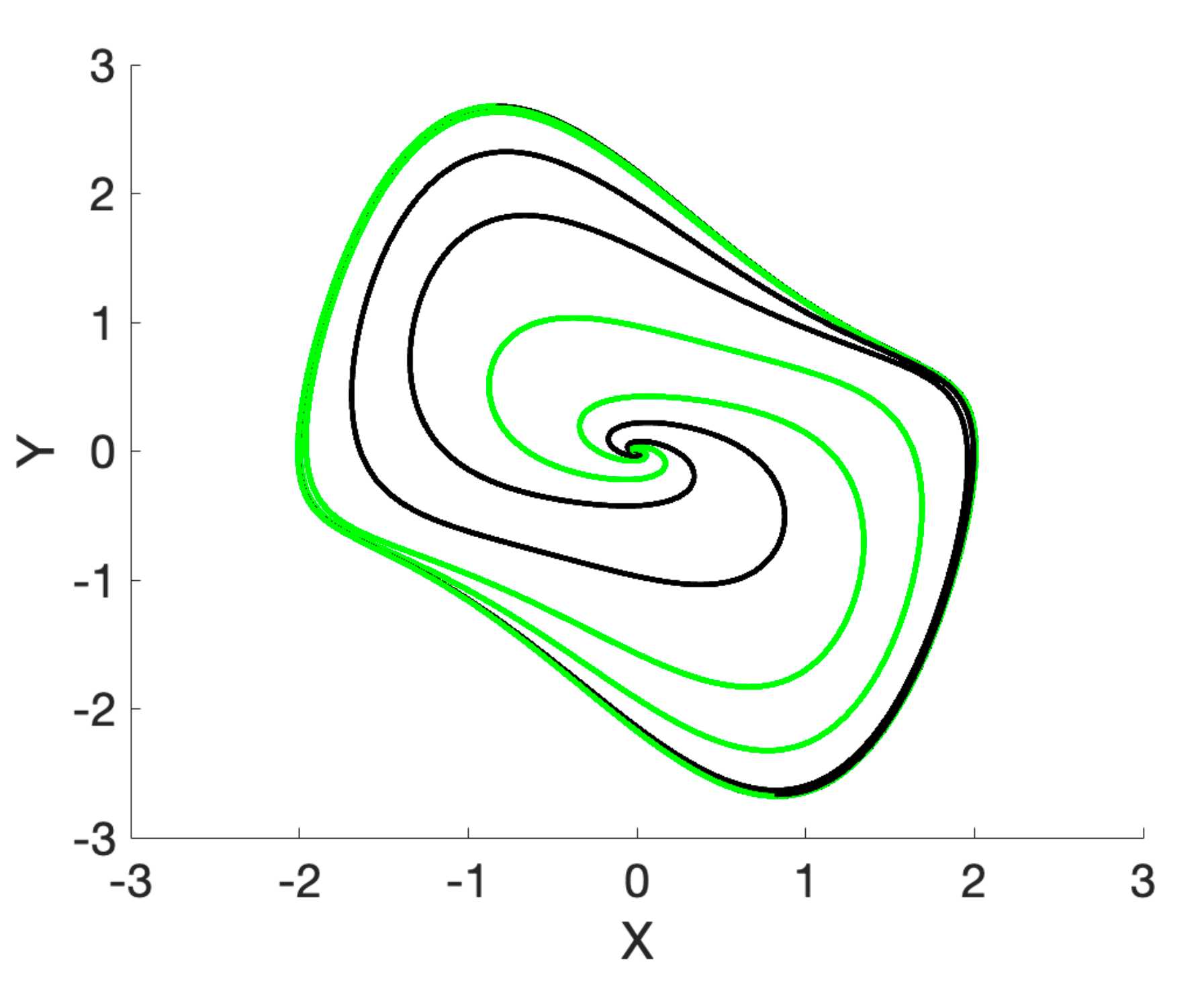}\\ 
 (c) & (d) 
\end{array}
$
\end{center}
\caption{(a) Plot of $W^u(O)$, and (b) Plot of $W^s(\Gamma)$ for 
$\eta=0.5$. One remark is that the phase portrait of Equation \eqref{eq:hamil} when projected onto the $(x,y)$ space is independent of noise strength so that one can recover one orbit for different noise values with the appropriate scaling in $(p,q)$ coordinates. In (c) we plot $W^u(O)$ (blue) and $W^s(\Gamma)$ (magenta) and note their transverse intersections in (x,y,p) space. The heteroclinic orbits are plotted in (d).}
\label{fig:stabunstab}
\end{figure}

\section{ Computing The Maslov Index and Conjugate Points
\label{sec:conjugatep}}

A conjugate point occurs along a trajectory in $W^u(O)$ when the tangent space to the invariant manifold at a point on that trajectory has a degenerate projection onto $(x,y)$ space, see Definition \ref{def:conj_pt}. Such points can be found by tracking the tangent space to $W^u(O)$ along trajectories in $W^u(O)$. 

\subsection{The Space of Lagrangian Planes} Individual tangent vectors will satisfy the linearized equations of Equation \eqref{eq:hamil}. Since Equation \eqref{eq:hamil} is Hamiltonian, the linearized system can be written in the form
\begin{equation}\label{eq:evol}
        \dot{U} = AU, \quad U \in R^{4}, 
    \end{equation}
    where $A=JB$, with
    \[ B =  \mat{ 
pf_{xx}+qg_{xx}& pf_{xy}+qg_{xy}&f_x&g_x\\
pf_{yx}+qg_{yx}& pf_{yy}+qg_{yy}&f_y&g_y\\
f_x&f_y & 1&0\\
g_x&g_y&0&1
 }
\]
evaluated on a solution $\left( x(t),y(t),p(t),q(t)\right) $ of Equation \eqref{eq:hamil}, and $J$ is the usual $4\times 4$ symplectic matrix
\[J=\begin{pmatrix}
0 & I_2\\
-I_2 & 0
\end{pmatrix}\]
with $I_2$ the $2\times 2$ identity. Note that $B$ is symmetric, which is a consequence of the Hamiltonian structure. 

Tangent spaces to invariant manifolds in a Hamiltonian system have a special property, called Lagrangian.

\begin{definition}\label{def:lagrange} A 2D subspace $\Pi \subset \R^4$ is said to be {\em Lagrangian} if $\langle JX,Y \rangle=0$ for all $X,Y \in \Pi$.\end{definition}

The collection of all 2D Lagrangian subspaces of $\R^4$ is called the space of Lagrangian planes, and denoted $\Lambda(2)$. It can be viewed as a submanifold of the Grassmannian of 2-planes in $\R^4$. It has the amazing property  that its fundamental group is the integers, $\pi_1(\Lambda(2)) = \mathbb{Z}$. This allows one to define a phase in $\Lambda(2)$ and the standard definition of the Maslov Index is that it counts the winding of this phase. The fundamental group of the full Grassmannian is $\mathbb{Z}_2$ and so that has no winding index, and thus the Lagrangian property is critical in making the Maslov Index work. We want to relate this characterization of the Maslov Index as a winding number to the conjugate point definition given in Definition \ref{def:conj_pt}. 

An index, such as the Maslov Index, can be represented by an intersection number. The simplest analogue here is the winding of a curve in the punctured plane corresponding to the intersection number with a half-line (such as the positive $y$-axis). For the object that will represent the Maslov Index through an intersection number with a curve in $\Lambda(2)$, we first need to define the Dirichlet subspace.

\begin{definition} The Dirichlet subspace of $\R^4$ is \[
    \mathcal{D}=\{(u,w)\in \R^4 :u=0\}. \]
 \end{definition}   
 \noindent It is not hard to check that    $\mathcal{D} \in \Lambda(2)$, i.e., it is Lagrangian. The Dirichlet subspace is key as a conjugate point occurs exactly when the tangent space to $W^u(O)$ at a point of a  trajectory non-trivially intersects $\mathcal{D}$. The train of $\mathcal{D}$, which we denote $\mathcal{S}(\mathcal{D})$  is the set of 2D subspaces in $\R^4$ that non-trivially intersect $\mathcal{D}$. Although it is an awkward way to state it, a conjugate point occurs exactly when the tangent space to $W^u(O)$ intersects the train $\mathcal{S}(\mathcal{D})$ in $\Lambda(2)$.  
 
 For our case, the Maslov Index, as a winding number, can be realized as the intersection of the curve of tangent spaces along the trajectory in $W^u(O)$ with $\mathcal{S}(\mathcal{D})$. But that is exactly the number of conjugate points.

\subsection{Pl{\"u}cker Coordinates}
Coordinates on the space of planes can be given that allow us to track the tangent space to $W^u(O)$ along a trajectory. The key is to form the 
Pl{\"u}cker coordinates \correction{comment 19}{\cite{postnikov1982lectures, colonius2014dynamical}} of an individual plane (2D subspace) in $\R^4$.

  Let $\Pi$ be a plane spanned by $v_1$ and $v_2$ with:
    \eq{
    v_1=\mat{v_{11}\\v_{12}\\v_{13}\\v_{14}} \textrm{and }
    v_2=\mat{v_{21}\\v_{22}\\v_{23}\\v_{24}}
    }{\notag}
    We set:
    \eq{
    \rho_{ij} = \begin{vmatrix}
    v_{1i} & v_{1j}\\
    v_{2i} & v_{2j}
    \end{vmatrix} = dx_i \wedge dx_j(v_1,v_2), \quad 1\leq i,j \leq 4, \quad i \neq j.
    }{\notag}
    For our particular problem, $\rho_{12}=dx \wedge dy,\rho_{13}=dx \wedge dp,\rho_{14}=dx \wedge dq,\rho_{23}=dy \wedge dp,\rho_{24}=dy \wedge dq,\rho_{34}=dp \wedge dq$.
    
How the Pl{\"u}cker coordinates of a plane vary can then be captured by an ODE governing the variation in time of the plane's Pl{\"u}cker coordinates. This can be calculated using the properties of differential forms from Equation \eqref{eq:evol} with $U=(dx,dy,dp,dq)$. 

    \eq{
\frac{d \hat{U}}{dt} = B(x(t),y(t),p(t),q(t))\hat{U},
}{\label{eq:conj2}}
where $\hat{U}=(\rho_{12},\rho_{13},\rho_{14},\rho_{23},\rho_{24},\rho_{34})$ and 

\eq{
B(x,y,p,q) = \mat{ f_x+g_y& 0&1 &-1&0&0\\
-pf_{xy}-qg_{xy}&0&-g_x&f_y&0&0\\
-pf_{yy}-qg_{yy}&-f_y&f_x-g_y&0&f_y&1\\
pf_{xx}+qg_{xx}&g_x&0&-f_x+g_y&-g_x&-1\\
pf_{yx}+qg_{yx}&0&g_x&-f_y&0&0\\
0&pf_{xy}+qg_{xy}&-pf_{xx}-qg_{xx}&pf_{yy}+qg_{yy}&-pf_{xy}-qg_{xy}&-f_x-g_y}.
}{\notag}
Note that this is evaluated along a trajectory $(x(t),y(t),p(t),q(t))$ which we are taking to lie in $W^u(O)$. To restrict to the coordinates of Lagrangian planes, we note that a plane is Lagrangian if (and only if) 
\eq{\rho_{13} +\rho_{24}=0.}{\notag}

A conjugate point can be conveniently described in Pl{\"u}cker coordinates.

\begin{lemma} The time $t=\tau$ is a conjugate point for a trajectory $z(t)=(x(t),y(t),p(t),q(t))$ in $W^u(O)$ if $\rho_{12}(\tau)=0$ for the Pl{\"u}cker coordinates of $T_{z(\tau)}W^u(O)$.
\end{lemma}
The trajectories in $W^u(O)$ are parameterized by angles $\theta$ that determine a point on the simple closed curve $\mathcal{K}\subset W^u(O)$. The methodology for finding conjugate points along a trajectory that passes through $\mathcal{K}(\theta)$ is as follows:
\begin{enumerate}
\item Compute the trajectory $z(t)$ backwards from $\mathcal{K}(\theta)$ until close to the fixed point at $O$.
\item Form Pl{\"u}cker coordinates of the unstable subspace of Equation \eqref{eq:hamil} at $O$ and initialize Equation \eqref{eq:conj2} with these coordinates at the time and point found in the first step.
\item Integrate Equation \eqref{eq:conj2} forward in time and find the values of $t$ where $\rho_{12}=0$.
\end{enumerate}

\begin{figure}[tbp]
\begin{center}
$
\begin{array}{cc}
\includegraphics[scale=0.25]{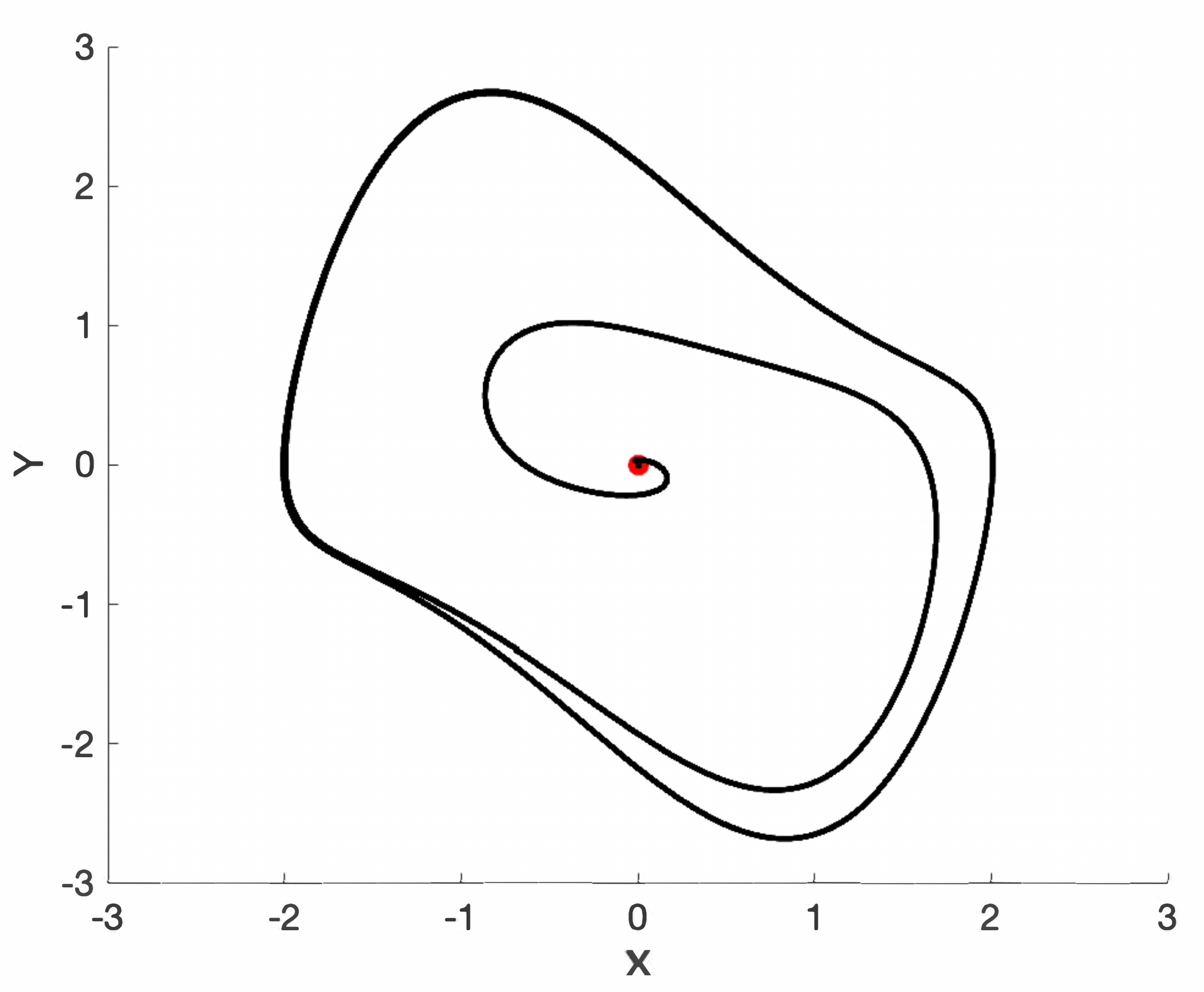} &
 \includegraphics[scale=0.25]{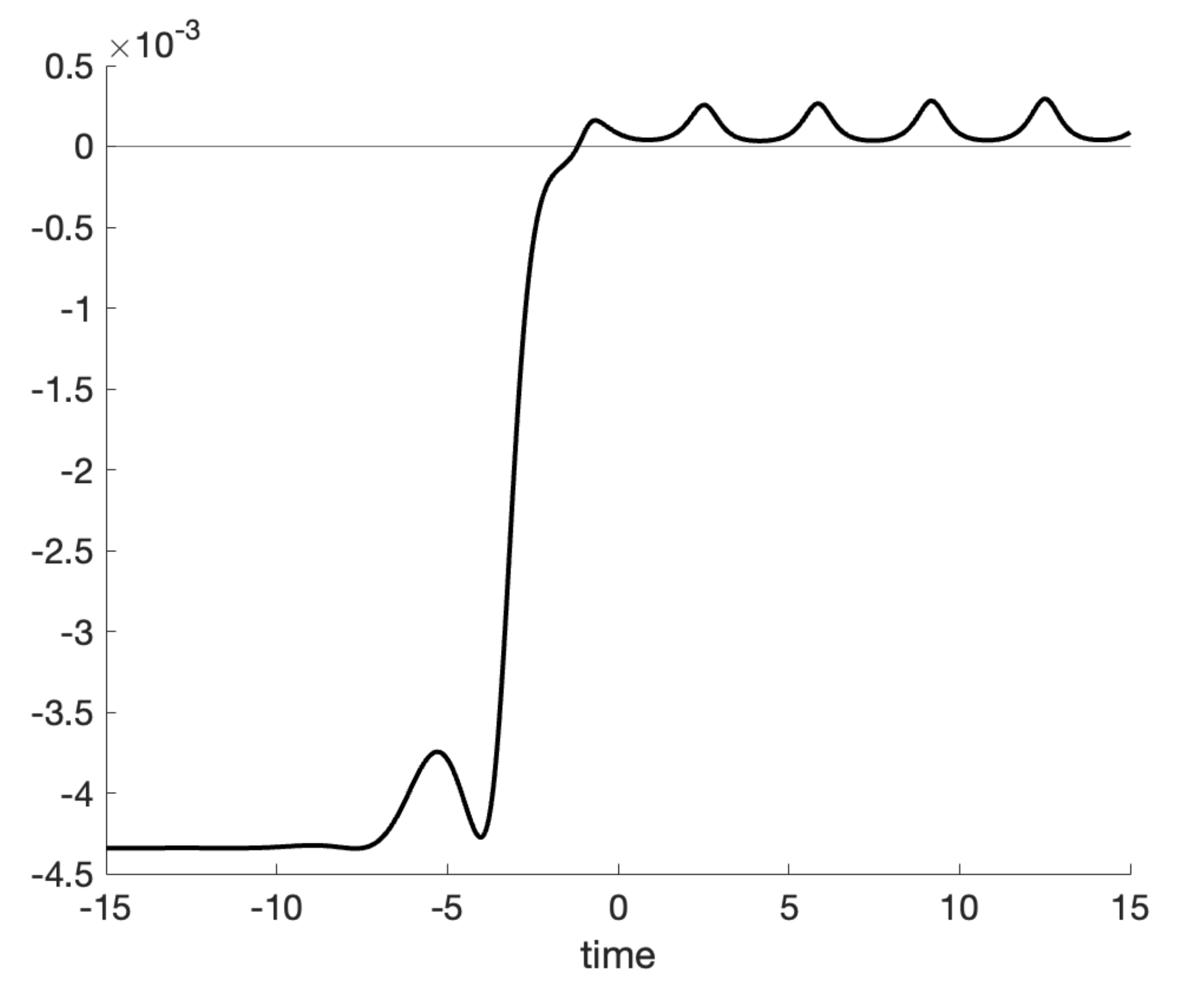}\\ (a) & (b)\\
\includegraphics[scale=0.25]{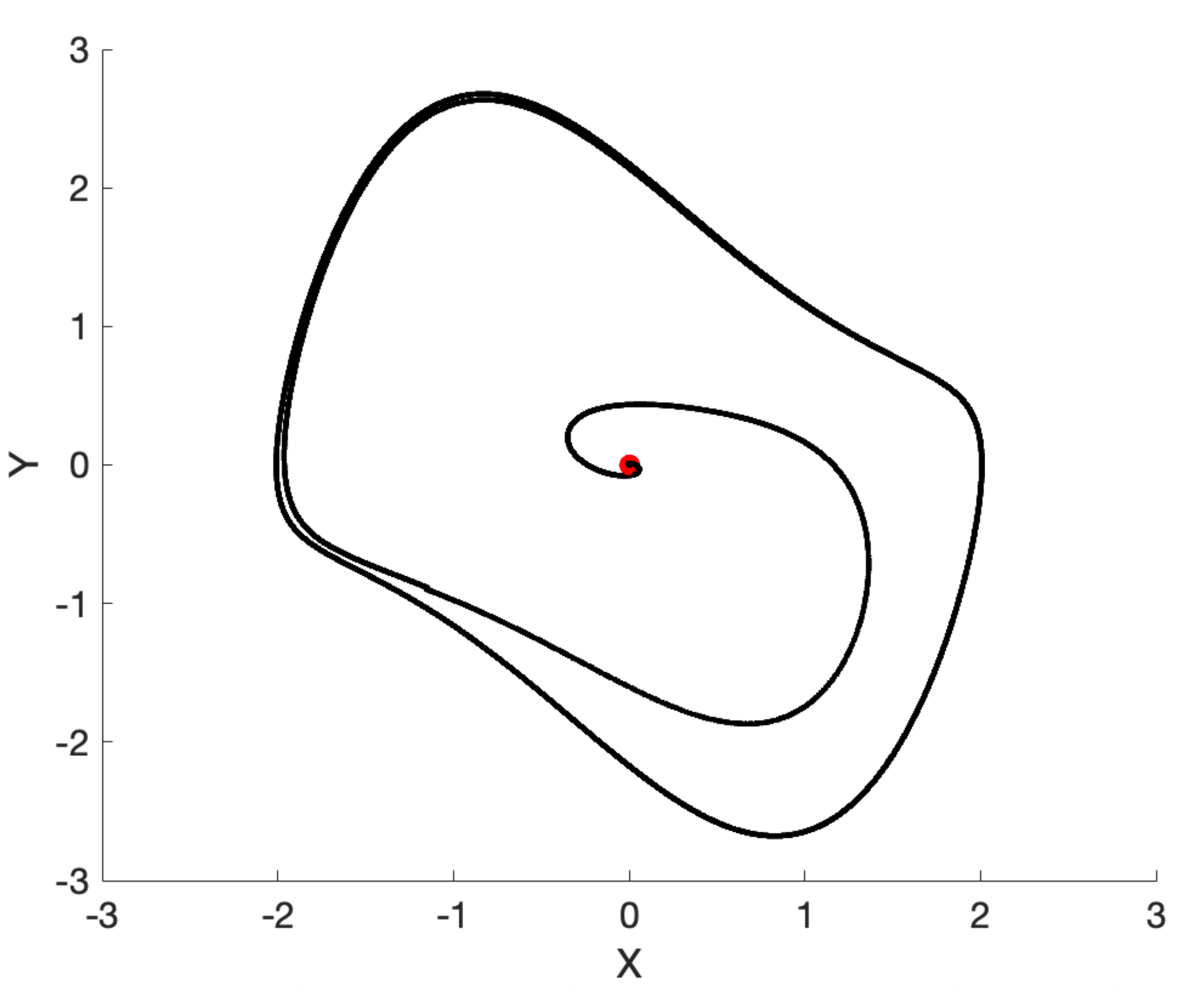} &
 \includegraphics[scale=0.25]{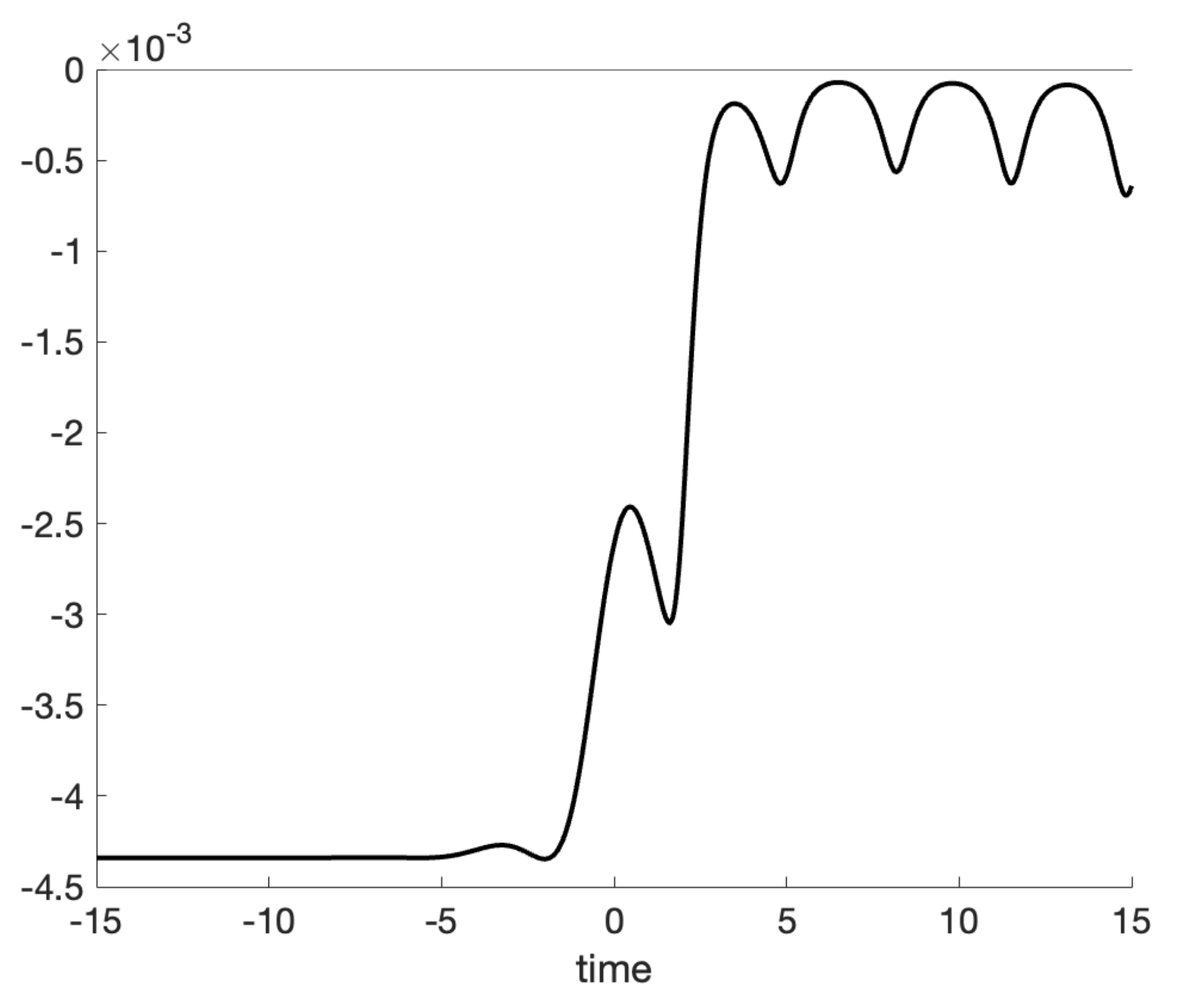}\\ (c) & (d)\\
\end{array}
$
\end{center}
\caption{(a) Plot of a heteroclinic orbit of Equation \eqref{eq:hamil} for 
$\eta=0.5$. (b) Plot where we detect a conjugate point (where $\rho_{12}=0$) for the heteroclinic orbit on the top left. (c) Plot of another heteroclinic orbit of Equation \eqref{eq:hamil} for the same parameters and (d) reveals that there are no conjugate points for the associated heteroclinic orbit on the left for the time interval that we specified. 
}
\label{fig:Hetero-Maslov}
\end{figure}

In Figure \ref{fig:Hetero-Maslov} we illustrate this for two key trajectories of IVDP, namely the heteroclinic orbits that we computed from Section \ref{sec:heteroclinics}. The plots on the right indicate where we detect a conjugate point ($\rho_{12}=0$) for the time interval that we specified, which we tracked from a small neighborhood of $O$ to the periodic orbit $\Gamma$ of Equation \eqref{eq:hamil} for each of the associated heteroclinic orbits. 

These computations confirm that one has Maslov Index 0 (the one shown in panel (c)) and the other has Maslov Index 1 (panel (a)).  The former is thus $\mathcal{H}_1$ and the latter $\mathcal{H}_2$.

\section{Trajectories Exiting  Over the Periodic Orbit
\label{sec:river}}

Under assumption (A3), there are two heteroclinic connections, $\mathcal{H}_1$ and $\mathcal{H}_2$, between the fixed point and $\Gamma$. We set $\theta_1$ and  $\theta_2$ to be the angles for points on $\mathcal{K}$ at which $\mathcal{H}_1$ and $\mathcal{H}_2$ pass through $\mathcal{K}$ respectively. These are depicted in Figure \ref{fig:Hetero_River1}. 

According to Lemma \ref{lemma:exit}, there are $\theta$-values near $\theta_1$ for which the associated trajectories pass over $\Gamma$, or more precisely through $\mathcal{T}_{\Gamma}$, and so are exit trajectories. The angles $\theta_1$ and $\theta_2$ divide the circle into two parts. Without loss of generality, we can assume that these exit trajectories correspond to $\theta$ values between $\theta_1$ and $\theta_2$.  We will make the further assumption that all trajectories with angles between $\theta_1$ and $\theta_2$  exit $\Gamma$.
\begin{description}
    \item[(A5)]  Every trajectory associated with angles $\theta \in (\theta_1, \theta_2)$ crosses $\Gamma$ when projected on the $(x,y)$ space.
\end{description}

While this seems like a strong assumption, it captures the situation we are imagining. We expect that between two heteroclinics the unstable manifold will leak out, but the complexity of the problem makes that hard to prove in general and so we make it as an assumption that can be verified numerically in examples as needed.  
\begin{figure}[tbp]
\begin{center}
\includegraphics[height=6cm]{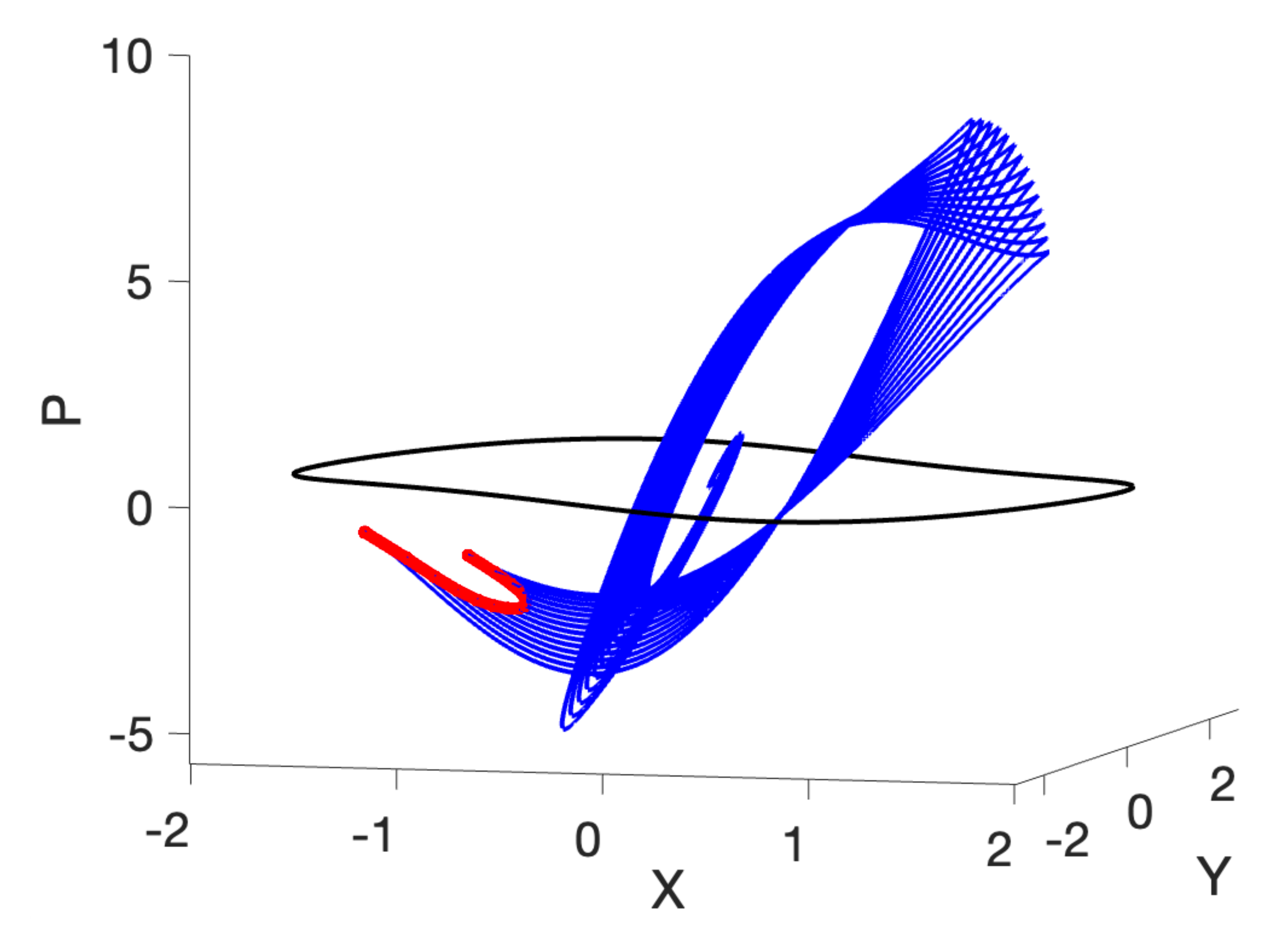}
\end{center}
\caption{Trajectories associated with $\mathcal{R}$ (in blue) reaching $\Gamma$ (in black) in $(x,y,p)$ space. The red curve represents the intersection of the torus $\mathcal{T}_{\Gamma}$ and trajectories associated with the river $\mathcal{R}$ projected in $(x,y,p)$ space.}
\label{fig:River3D}
\end{figure}

\subsection{The River}
 The term {\em River} will be used to describe the set of trajectories on $W^u(O)$ that cross $\Gamma$ with $\theta$ values between $\theta_1$ and $\theta_2$. The curves $\mathcal{H}_1$ and $\mathcal{H}_2$ form the ``banks'' of the river.  In the following definition, $Z$ is a trajectory $z(t)$ for $t \in (-\infty,0]$ satisfying Equation \eqref{eq:hamil}.

\begin{definition}
The \textbf{full River} $\mathcal{R}$ is defined as 

\begin{equation}
\mathcal{R}= \{Z | z (\tau) \in \mathcal{K}(\theta)\textrm{\ for\ some}\ \theta_1< \theta< \theta_2 \ \textrm{and for some}\ \tau <0 \ \textrm{and further} \  z(0) \in \mathcal{T}_{\Gamma}\}.
\notag
\end{equation}
\end{definition}
\noindent The full river $\mathcal{R}$ corresponding to the IVDP is depicted in Figure \ref{fig:conj_on_river} (a). Note that, we are parameterizing the trajectories so that each one crosses the exit torus $\mathcal{T}_{\Gamma}$ at $t = 0$. 

\subsection{River Trajectories as Minimizers}

We also define a sub-river $\tilde{\mathcal{R}}\subset \mathcal{R}$ that consists of trajectories with zero Maslov Index.

\begin{equation}
  \tilde{\mathcal{R}}:= \{Z\in \mathcal{R}|\ m(Z) = 0\}.  
\end{equation}

\noindent In Figure \ref{fig:conj_on_river}, both in a) and b), we plot the conjugate points (in red) for several paths in the river. 

\begin{figure}[tbp]
\begin{center}
$
\begin{array}{cc}
\includegraphics[scale=0.32]{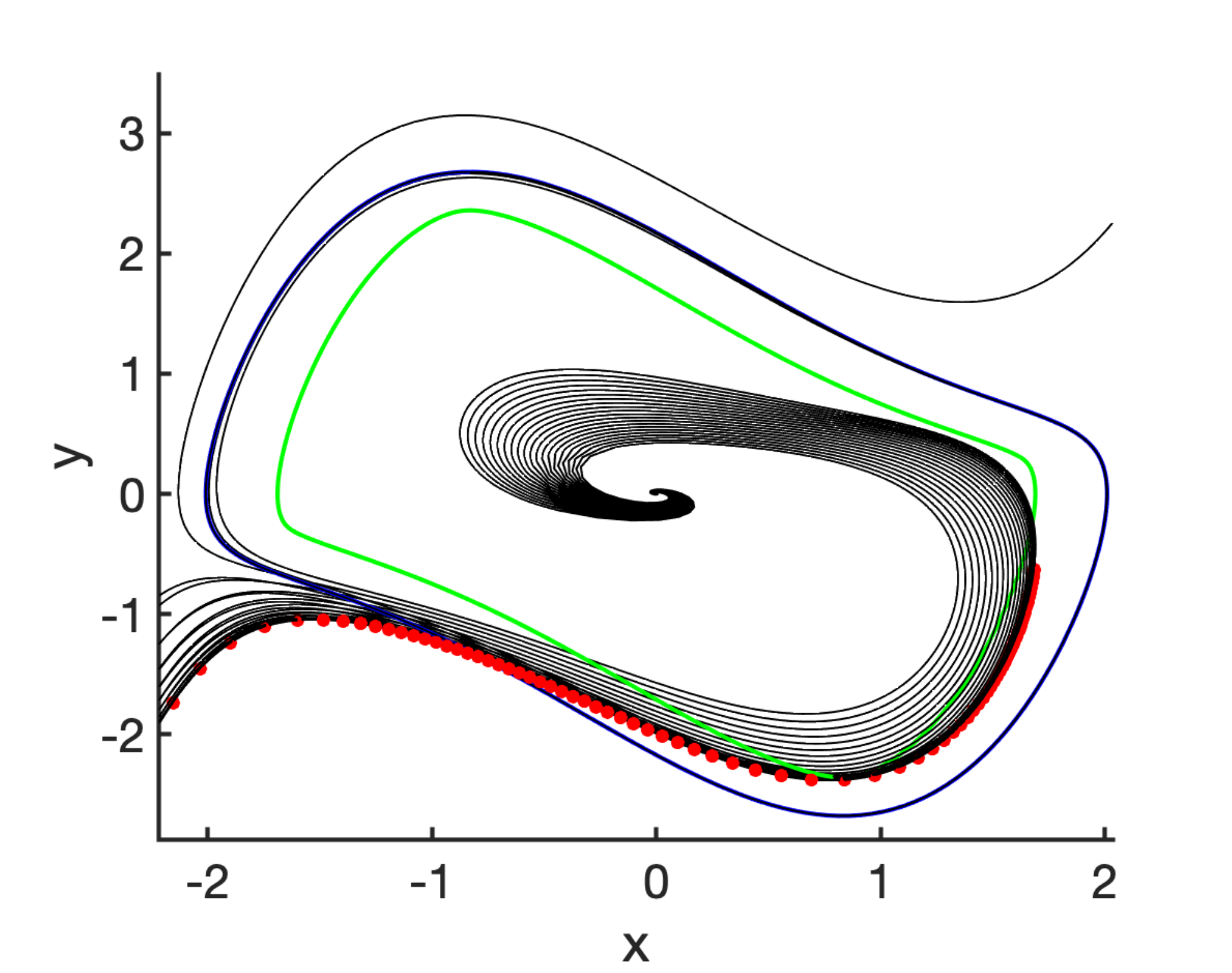} & \includegraphics[scale=0.32]{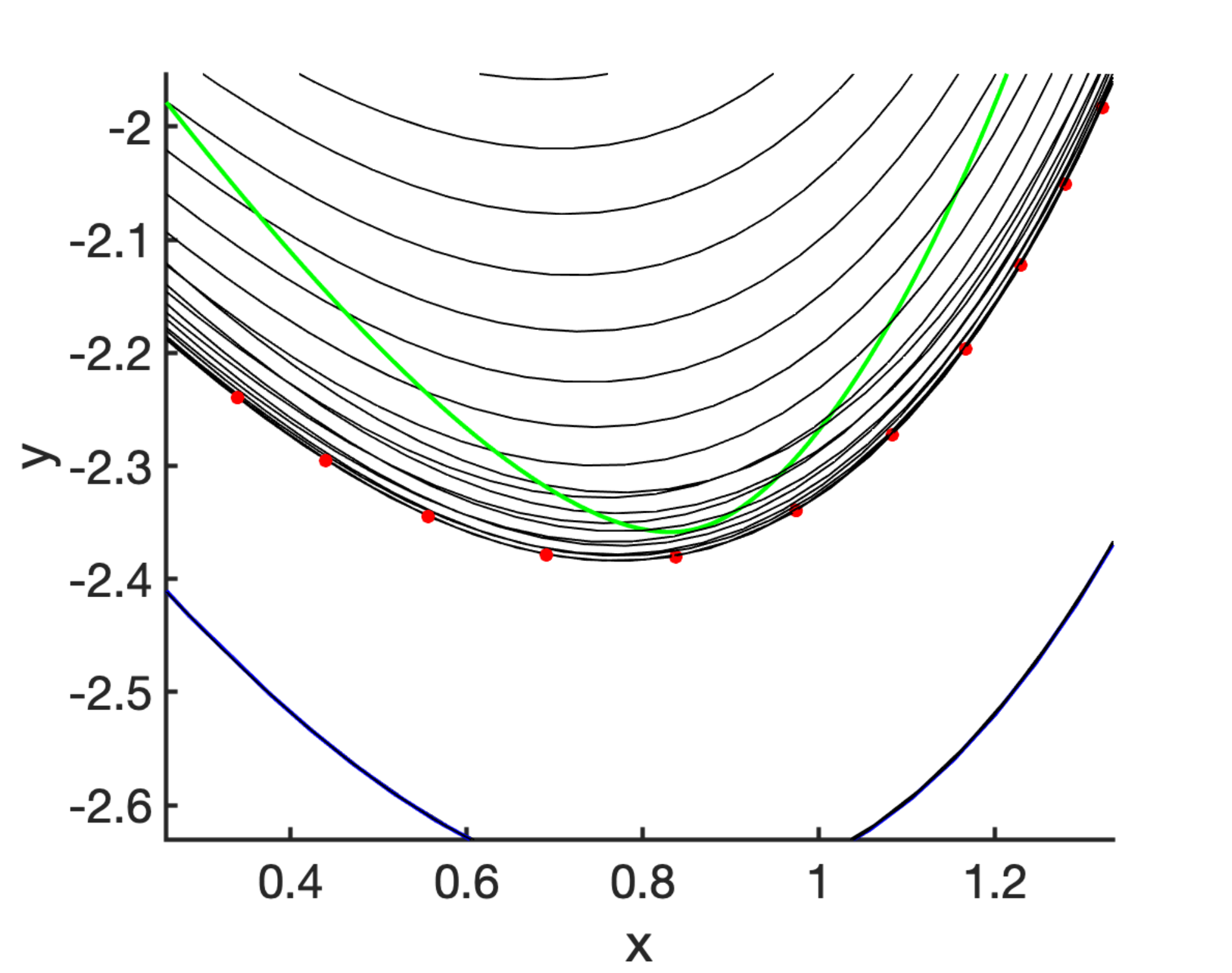} 

\end{array}
$
\end{center}
\caption{(a) Plot of the the conjugate points for 201 paths in the river along with every 10th path. (b) Zoomed in plot that shows the conjugate points occur after the paths have crossed the 0.32 border. 
}
\label{fig:conj_on_river}
\end{figure}

The main theorem states that the trajectories in $\tilde{\mathcal{R}}$ are minimizers of the Friedlin-Wentzell functional given their respective boundary value.

\begin{theorem}\label{thm:main}
Every trajectory in $\tilde{\mathcal{R}}$ is a local minimizer of the Friedlin-Wentzell action functional among the trajectories with the same boundary value $(x(0),y(0)) \in \Gamma$.
\end{theorem}

\begin{proof}
Let $z= (x,y,p,q)$ be an element of $\tilde{\mathcal{R}}$. Let the columns of $V$ form a basis for the solution space of $V' = A^* V$ that belongs to the unstable manifold at the origin, where $A^*$ is the linearization of Equation \eqref{eq:hamil} about $z$. We can write $A^*$ as
\eq{
A^* = \mat{B&Id\\ -A&-B^T}
}{\notag}
where 
\eq{
A = \mat{pf_{xx}+qg_{xx}&pf_{xy}+qg_{xy}\\ pf_{yx}+qg_{yx} & pf_{yy}+qg_{yy}},\quad B = \mat{f_x&f_y\\ g_x&g_y}. 
}{\notag}
The second variation of the Friedlin-Wentzell functional is given by
\eq{
\delta^2 I[h^T,h] &= \int_{-\infty}^0 h^T(B^TB-A)h-h^TB^T\dot h - \dot h^TBh+\dot h^Th dt,
}{\notag}
where $\|h\| = 1$ is the direction of perturbation and $^.:=\frac{d}{dt}$. We note that $h(0)=0$ and that $h$ and its first derivative decays to zero at exponential rate as $x\to -\infty$, and that $A$ and $B$ are uniformly bounded since they asymptotically decay to constants states. Thus the integrals given in the definition of $I[h]$ and $\delta^2 I[h^T,h]$ converge. 
Following a standard calculation (for example see \cite{liberzon_calculus_2011}), define $V_1$ and $V_2$ so that $V = (V_1^T,V_2^T)^T$. Then
\eq{
\dot V_1 = BV_1+V_2,\quad
\dot V_2 = -AV_1-B^TV_2.
}{\notag}
Since there are no conjugate points associated with $z$, the matrix $V$ is full rank throughout its domain, and thus $V_1$ is invertible. Define $W = -V_2V_1^{-1}.$
Then 
\eq{
\dot W&= -\dot V_2V_1^{-1}+V_2V_1^{-1}\dot V_1 V_1^{-1}\\
&= A- B^TW-W(BV_1+V_2)V_1^{-1}\\
&= A-B^TW-WB+W^2.
}{\notag}
Note that if the initial data is symmetric, then $W$ is symmetric. To see this, if $W$ is a solution that satisfies $W(t_0) = W_0$ where $W_0^T = W_0$ , then $W^T(t)$ is also a solution of the \correction{comment 20}{Riccati} equation that has the same initial conditions. By uniqueness of solutions, $W(t) = W^T(t)$. 

We will shortly use the following fact. Note that for any $C^1$ matrix valued function $W:(-\infty,0]\to \R^{n\times n}$ that
\eq{
0&= hWh^T|_{-\infty}^{0}\\
&= \int_{-\infty}^0 \frac{d}{dx}(hWh^T)dx\\
&= \int_{-\infty}^0 h'Wh^T+hW'h^T + hW(h')^T.
}{\notag}
Now 
\eq{
\delta^2 I &= \int_{-\infty}^0 h^T(B^TB-A)h-h^TB^Th-\dot h^TBh+\dot h^Th\\
&= \int_{-\infty}^0 h^T(\dot W+B^TB-A)h+h^T(W-B^T)h+\dot h^T(W-B)h+\dot h^T\dot h\\
&= \int_{-\infty}^0 h^T(B^TB-B^TW-WB+W^2)h + h^T(W-B^T)h+\dot h^T(W-B)h+\dot h^T\dot h\\
&= \int_{-\infty}^0 ([B-W]h-\dot h)^T([B-W]h-\dot h).
}{\notag}
Thus, the second variation is non-negative if there are no conjugate points. We now use a perturbation argument to show that the second variation is actually bounded from below. 
Suppose that for some \correction{comment 17}{$\mu > 0$} there are no conjugate points for the system $V'=  D_{\mu}V$, where
\eq{
D_{\mu} := \mat{ B & (1-\mu)Id\\ -A-\mu Id&-B^T}.
}{\notag}
Then we have
$\delta^2 I[h^T,h] \geq \mu \int h^Th+\dot h^T\dot h\geq \mu \|h\|^2$. Thus, to show that $z$ is a local minimum of $I$, we only need show that there exists $\mu > 0$ such that $V'=D_{\mu}V$ has no conjugate points. 

Now as $t \to -\infty$:
\eq{ D_{\mu} \to D_{\mu}^{\infty}:= \mat{ B_{\infty} & (1-\mu)Id\\ -A_{\infty}-\mu Id&-B_{\infty}^T}
}{\notag}
where $A_{\infty}$ and $B_{\infty}$ are $A$ and $B$ evaluated at $(0,0,0,0)$, respectively.
 We rewrite $V'=D_{\mu}V$ as an autonomous system:
\begin{equation}\label{eq:aut1}
\begin{split}
V'=&D_{\mu}(\tau)V\\
\tau'=&1,
\end{split}
\end{equation}
$(V, \tau) \in \R^4 \times (-\infty,0]$. 
Equation \eqref{eq:aut1} induces a flow on $\Lambda(2) \times (-\infty,0]$, where $\Lambda(2)$ is the space of Lagrangian 2 planes in $\R^4$, with the associated equation:
\begin{equation}\label{eq:aut2}
\begin{split}
v'=&d_{\mu}(v,\tau)\\
\tau'=&1,
\end{split}
\end{equation}
for some function $d_{\mu}(v,\tau)$, where $v \in \Lambda(2)$. Note here that the perturbed system is also a linear Hamiltonian system. 

Both Equations \eqref{eq:aut1} and \eqref{eq:aut2} can be compactified, see \cite{Wieczorek_2021}, by a map $\sigma:(-\infty,0] \to [-1,0]$. Setting $s=\sigma(\tau)$, Equation \eqref{eq:aut1} becomes
\begin{equation}\label{eq:aut3}
\begin{split}
V'=&D_{\mu}(h(s))V\\
s'=&g(s),\\
\end{split}
\end{equation}
where $h(s):= \sigma^{-1}(s)$ and $g(s):= \sigma'(h(s))$, and Equation \eqref{eq:aut2} becomes
\begin{equation}\label{eq:aut4}
\begin{split}
v'=&d_{\mu}(v,h(s))\\
s'=&g(s),
\end{split}
\end{equation}
with $g(-1)=0$.

Note that $s=-1$ corresponds to $\tau=-\infty$ and so 

\begin{equation*}
\begin{split}
D_{\mu}(h(s=-1)) = &D^{\infty}_{\mu},\\
d_{\mu}(v,h(s=-1)) = &d^{\infty}_{\mu}(v).
\end{split}
\end{equation*}

When $\mu=0$, we have $D_{0}^{\infty}$ has $2$ complex conjugate eigenvalues with negative real part and $2$ complex conjugate eigenvalues with positive real part.
\noindent Now, $s=-1$ is invariant and the 2D unstable subspace $V^u_0$ of $V'=D^{\infty}_0V$ becomes a fixed point $v ^u_0$ of 
\begin{equation}\label{eq:pert2}
v'=d^{\infty}_0v.
\end{equation}
Moreover it perturbs to a fixed point $v ^u_{\mu}$, if $\mu >0$ is sufficiently small, of
$
v'=d^{\infty}_{\mu}v,
$
since it is attracting in Equation \eqref{eq:pert2}.

Next, consider equation \eqref{eq:aut4} on $\Lambda(2) \times [-1,0]$ when $\mu=0$, then $v^u_0$ is a fixed point with 3D stable manifold, which lies inside $\{s=-1\}$, and a 1D unstable manifold. The same will hold for sufficiently small $\mu >0$. The 1D unstable manifold is the object we want. Moreover, by construction, it varies smoothly in $\mu$.
Thus if $V' = D_{0}V$, $t \in (-\infty,0]$ produces no conjugate points, then $V' = D_{\mu}V$ also does not, as long as $\mu >0$ is small enough.  \end{proof} \correction{comment 7}{}

It needs to be emphasized here that we only expect these trajectories to be {\em local} minimizers. There will be infinitely many trajectories in $W^u(O)$ that cross $\Gamma$ at a fixed $(x,y)\in \Gamma$ and an infinite sequence of them will consist of local minimizers. Moroever, they will have decreasing action value and the limit will be the action value of $\mathcal{H}_1$. 

This is the cycling phenomenon known to occur when there is a periodic boundary, see \cite{day_exit_1996}. the MPEP is the heteroclinic and, as it is then the mode of the escaping trajectories, it will enforce cycling of trajectories that escape. Our point is that this will only occur in the limit of vanishing noise. Moreover, our contention is that the trajectories in $\tilde{\mathcal{R}}$ play a key guiding role for the (noisy) trajectories that escape in small but not vanishing noise. 

The trajectories in $\tilde{\mathcal{R}}$ can be viewed as most probable paths of a constrained problem. If we consider the problem of finding the most probable path between $O$ and a particular point $(x,y)$ on $\Gamma$, then the trajectories on $\tilde{\mathcal{R}}$ will appear. In probabilistic terms, such a trajectory is a candidate for the most probable path when conditioned on exiting $\Gamma$ at exactly that point. 

If we further restrict the amount of cycling in the condition, then there will be a path in $\tilde{\mathcal{R}}$ that will be a global minimizer. To give an exact accounting of such a cycling condition is not straightforward and will not be taken up here.  Nevertheless, this idea should give some credence to our view that these trajectories play a key role in understanding escape through the periodic orbit.

\subsection{Pivot Point}\label{sec:pivot}
The river $\mathcal{R}$ and the 0-Maslov Index sub-river $\tilde{\mathcal{R}}$ can be characterized in terms of the points in the interval of angles: $\left [ \theta _1, \theta _2 \right ]$. We introduce a transition map 
\begin{equation}\label{eq:transition}
G: \mathcal{K}\left(\theta _1,\theta _2 \right )\to \mathcal{T}_{\Gamma},   
\end{equation}
which takes each point on that part of $\mathcal{K}\subset W^u(O)$ to the point where the trajectory through that point first crosses the periodic orbit, i.e., lies in the torus $\mathcal{T}_{\Gamma}$ (which recall is $H=0$ with $(x,y)\in \Gamma$). Since $\theta_1$ and $\theta_2$ are not included, we know that every trajectory does indeed cross $\Gamma$. 

We will refer to the image of $G$ as the {\em mouth of the river}. we would like to find a subset $J$ of $\mathcal{K}\left(\theta _1,\theta _2 \right )$ so that $G(J)=\tilde{\mathcal{R}}$ but also have this be an interval (in the angle). In general, this cannot be guaranteed and so we take a subset of  $\mathcal{K}\left(\theta _1,\theta _2 \right )$ by the following procedure.

By Lemma \ref{lemma:exit}, we know that if $\theta$ is sufficiently close to $\theta _1$ then $G(\mathcal{K}(\theta))\in \tilde{\mathcal{R}}$. Let $$\hat{\theta} = \inf \{\theta | G(\mathcal{K}(\theta))\in \tilde{\mathcal{R}}\}.$$ Then the set $\mathcal{K}(\theta_1, \hat{\theta})$ is non-empty and $G\left( \mathcal{K}(\theta _1,\hat{\theta}) \right)\subset \tilde{\mathcal{R}}$. The end-point $G\left( \mathcal{K}(\hat{\theta })\right)$ will not lie in $\tilde{\mathcal{R}}$. In fact, we have the following lemma.

\begin{lemma} \label{lemma:pivot}The trajectory emanating from $\mathcal{K}(\hat{\theta})$ will have Maslov Index equalling 1 and the conjugate point will occur as the trajectory crosses the periodic orbit, i.e., when it is in $\mathcal{T}_{\Gamma }$. 
\end{lemma}

The set $Q=G\left( \mathcal{K}(\theta _1,\hat{\theta}) \right)$ plays a key role. It has the following properties:
\begin{enumerate} 
    \item $Q$ is an infinite spiral in $\mathcal{T}_{\Gamma}$,
    \item Its projection onto the $(x,y)$-space is all of $\Gamma$.
    \item It is pinned at one end by the pivot point. 
\end{enumerate}
The other end of this curve is the heteroclinic, but that is not seen in $\mathcal{T}_{\Gamma}$ as it never reaches it. By Theorem \ref{thm:main}, every point in $Q$ is a (local) minimizer of the Freidlin-Wentzell action functional with fixed boundary condition, except the pivot point. 

In the following sections, we shall see that the set $Q$ plays a key role in determining the escape hatch. But it is still too large, because of Property 2 above. 

From the Monte-Carlo simulations, we see that the escape hatch is near the pivot point, and definitely does not extend around all of $\Gamma$. 
        
\section{A Perturbed Action
\label{sec:onsager}}
In order to understand why the noisy escape trajectories do not veer too far from the pivot point when crossing $\Gamma$, we need to calculate the energy required by a path to escape to higher order. This involves the Onsager-Machlup (OM) functional \correction{comment 19}{\cite{OM_Durr_Bach,Bach, PhysRev.91.1505}}, which becomes relevant when the noise is not necessarily small.

Our viewpoint is to use the OM action as a selection mechanism among the trajectories that we find as (local) minimizers of the FW functional, in particular, the trajectories in $\tilde{\mathcal{R}}$.

For IVDP, the set $Q$ is exactly the part of the mouth of the river corresponding to  $\tilde{\mathcal{R}}$. In the following, we will therefore not distinguish between these two objects and take $Q=\tilde{\mathcal{R}}\cap \Gamma$ 

\subsection{The Onsager-Machlup Functional}
The OM functional for a path $z=z(t)$ on an interval $[a,b]$, it is given by 
\begin{equation}\label{eq:OM}
    I_{\varepsilon}(z)= \int_{a}^b \frac{(\dot{z}-F(z))^2}{2} + \varepsilon \left(\grad \cdot F(z) \right) dt,
\end{equation}
where $\varepsilon $ is the noise coefficient as usual. 
Applying this to IVDP, we calculate
\[\grad \cdot F(z) = \grad \cdot F(x,y) = 2 \eta(x^2-1).\]

\noindent 
For IVDP, it holds that $\int_{\Gamma} \grad \cdot F(z) dt >0$. Hence the OM functional $I_{\varepsilon}$, with $\eps > 0$, penalizes trajectories that cycle around the periodic orbit. Since the heteroclinic orbits wind around a neighborhood of $\Gamma$ infinitely many times, $I_{\varepsilon}$ evaluated near a heteroclinic orbit will tend to infinity. Thus, when the OM perturbation is added to the Friedlin-Wentzell (FW) functional, the heteroclinic connections cease to be global minimizers. 

\subsection{A Selection Mechanism}
Suppose we have a family of (local) minimizers of the FW functional on an interval $[0,T]$
$$\mathcal{F}= \{ z(t)|z(0) \in A, z(T) \in B \}, $$ where $A$ and $B$ are sets in $\R^2$. 
We can attempt to find the global minimizer of the FW functional $S_T(z)$ over the trajectories in $\mathcal{F}$. But it may be that there is no global minimizer in $\mathcal{F}$.

Exactly this situation occurs if $A=C$, a small circle around the fixed point $O$ and $B=\Gamma$. In this case, if we look for minimizers of $S_T(z)$ over $\mathcal{F}$ with fixed $z(T)\in \mathcal{T}_{\Gamma}$, we will obtain exactly the set $\Tilde{\mathcal{R}}$, which has found these minimizers as trajectories of Equation \eqref{eq:hamil} and weeded out those with non-zero Maslov Index. The (global) minimizer over $\mathcal{F}$ can now be found by sorting through the action values and finding the member of $\mathcal{F}$ with least action. But such a trajectory will not exist since the action decreases as the paths tend to the heteroclinic, which is not in $\mathcal{F}$. This is back to the same issue that lies behind the cycling phenomenon, namely that there is no global minimizer of the FW functional which crosses $\Gamma$. Note that, in this example, the circle $C$ is used as a proxy for $z(t)\to O$ as $t\to - \infty$, and that if $C$ is small enough, i.e., close enough to $O$, then the difference in action value of a member of $\mathcal{F}$ from taking one circle $C$ over another is negligible. 

The idea then is to use the OM functional as a perturbation of the FW functional to select which of these paths is the {\em Most Probable Escape Path} for small but non-vanishing noise.

\subsection{Evaluating OM along FW minimizers}
Since the path $z=z(t)$ is independent of $\eps$, the OM functional $I_\eps (z)$ is linear in $\eps$, and
\begin{equation}\label{eq:OM_pert}
    \frac{\partial I_{\varepsilon}}{\partial \varepsilon} 
    =  \int_{0}^{T} \grad \cdot F(z) dt.
\end{equation}
The right hand side of Equation \eqref{eq:OM_pert} can \correction{comment 8}{be} evaluated on a trajectory in $\mathcal{R}$. Since $\tilde{\mathcal{R}}\subset \mathcal{R}$, the trajectories of interest are included. The result for IVDP is shown in Figure \ref{fig135}. 

\begin{figure}[htbp]
 \begin{center}
$
\begin{array}{lcr}
\includegraphics[scale=0.5]{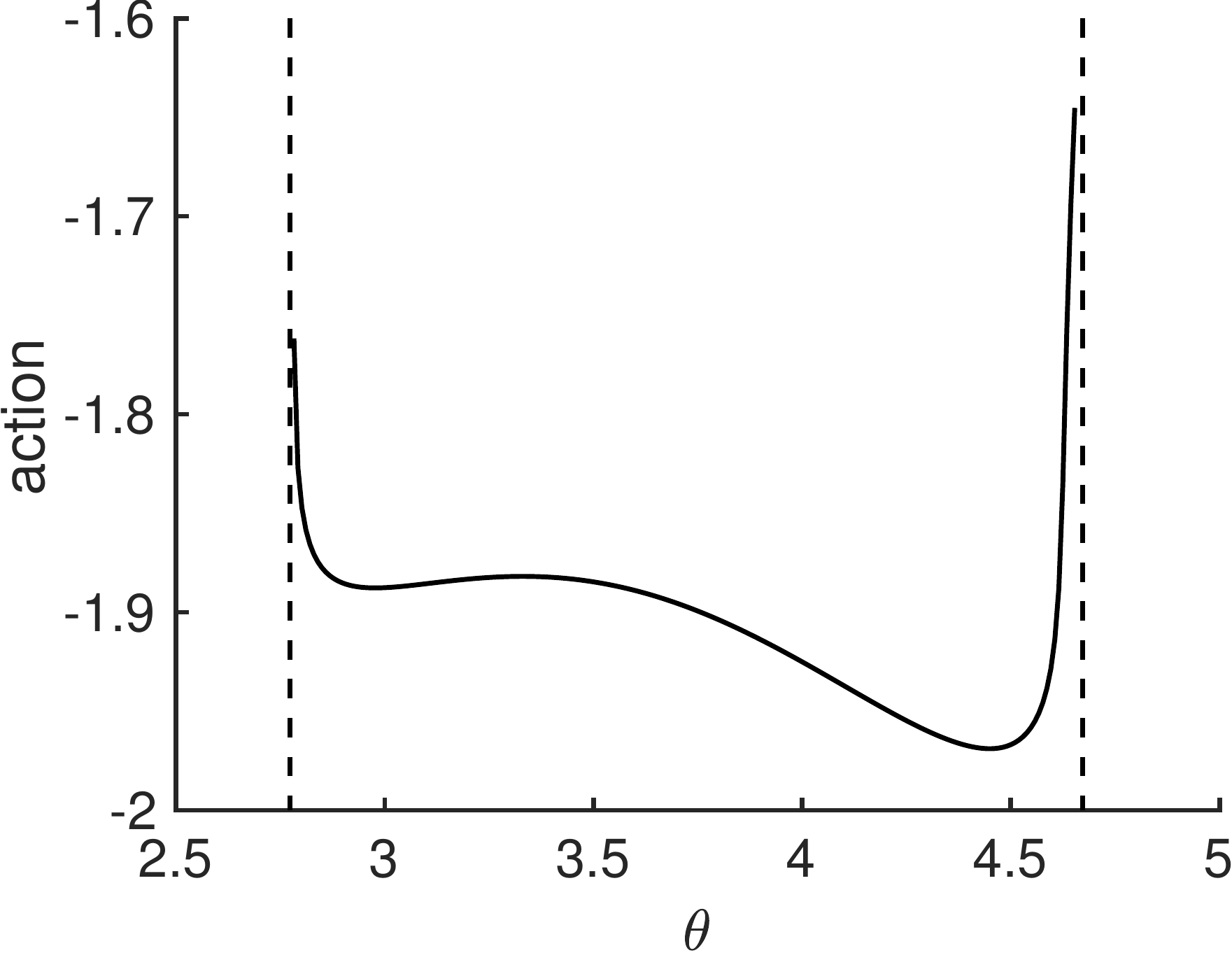}
\end{array}
$
\end{center}
\caption{Plot of the Onsager-Machlup functional  along trajectories in the full river $\mathcal{R}$.  The left and right dotted lines mark the value of $\theta$ corresponding to the unstable ($\theta_2$) and stable ($\theta _1$) heteroclinic connections respectively. 
 }
\label{fig135}
\end{figure}

One caveat is that the integral in the OM functional would not converge if computed along FW trajectories lying in $W^u(O)$ considered on the half-line $(-\infty, 0]$. This is because $\grad \cdot F(z)|_{z=0} \ne 0$. Integrating from $t=0$ and initiating on a small circle is designed to circumvent this challenge. In other words, we are computing the action between the boundary of a small neighborhood of the origin and the periodic orbit. The extra action one obtains by shrinking the neighborhood around the origin varies less and less among the orbits as the neighborhood gets small since $\grad \cdot F(z)$ converges to $-2\eta$ at a uniform exponential rate. Hence, in practice we obtain the most probable escape path predicted by the OM functional to within numerical precision by truncating orbits onto a finite domain.

\subsection{A Most Probable Escape Path according to OM} The action plot for IVDP is shown in Figure \ref{fig135}.  The minimum action occurs at $\theta_{\textrm{min}} \approx 4.44$. This is the object that we claim can be taken as an MPEP for this level of noise. It will be noise dependent and, were the noise to be decreased toward $0$, it would move toward $\theta _1$, i.e., the $\theta$ value at the (stable) heteroclinic.

In Figure  \ref{fig:MPEPC2}, we compare the projection onto the $(x,y)$-plane of the OM orbit corresponding to $\theta = 4.44$ with the most probable exit locations as given by our Monte Carlo simulations. The correspondence of the OM-selected path and the peak of the exit distribution is striking. Details about the Monte Carlo simulations and these comparisons are given in the next section (Section \ref{sec:mcsim}).

\section{Monte-Carlo Simulations
\label{sec:mcsim}}
The analysis we have carried out is aimed at finding a most probable path of escape for noisy trajectories through a periodic orbit that forms the boundary of the basin of attraction of the attracting fixed point. The work has been predicated on the notion that for small, but non-vanishing, noisy trajectories that escape will not exhibit cycling but rather find an ``escape hatch" at a specific part of the periodic orbit. Moreover, they will choose to leave the basin of attraction without overly cycling, at least not near the boundary (periodic orbit).

\begin{figure}[ht]
\begin{center}
$
\begin{array}{cc}
\includegraphics[scale=0.22]{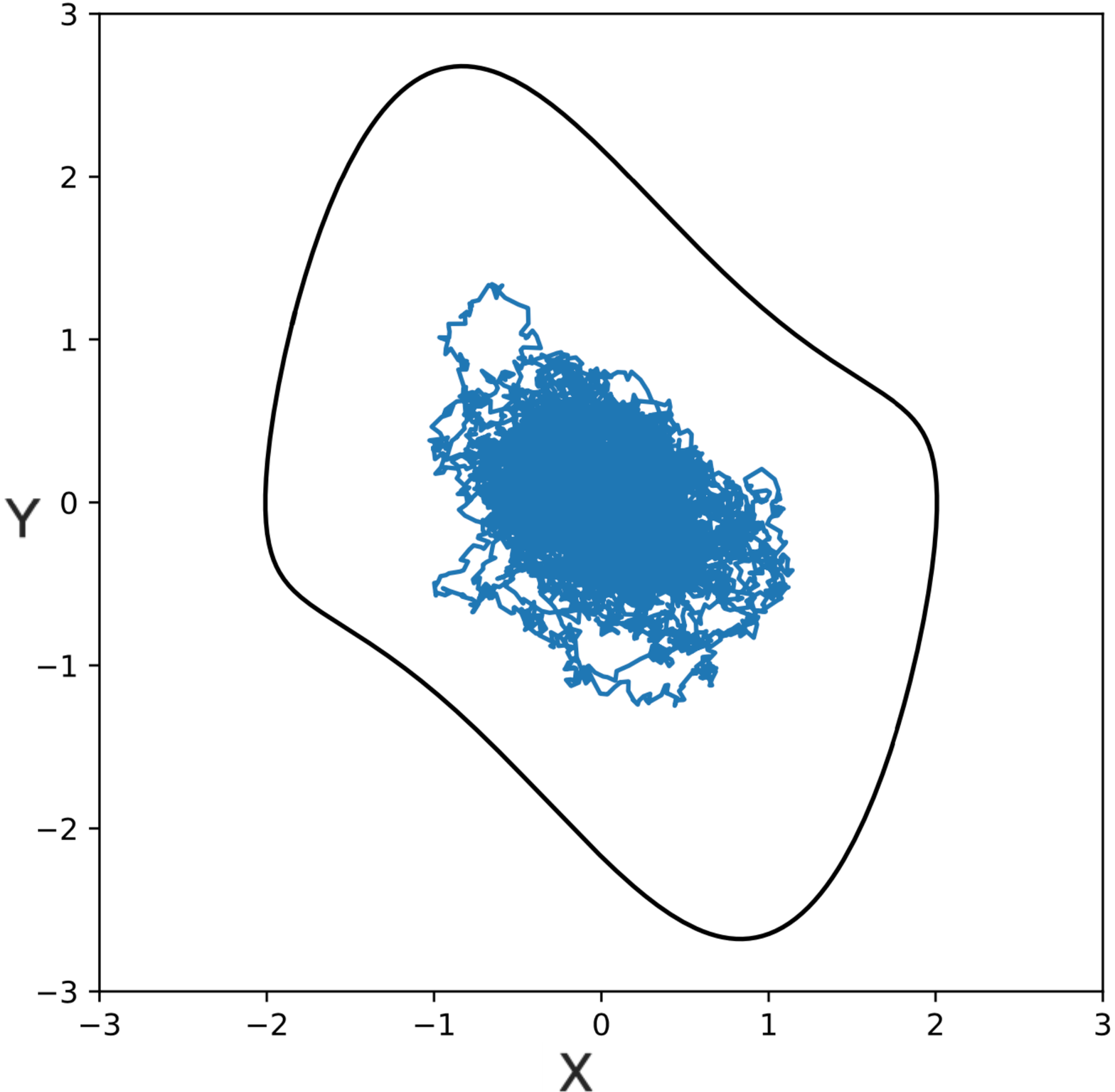} & \includegraphics[scale=0.22]{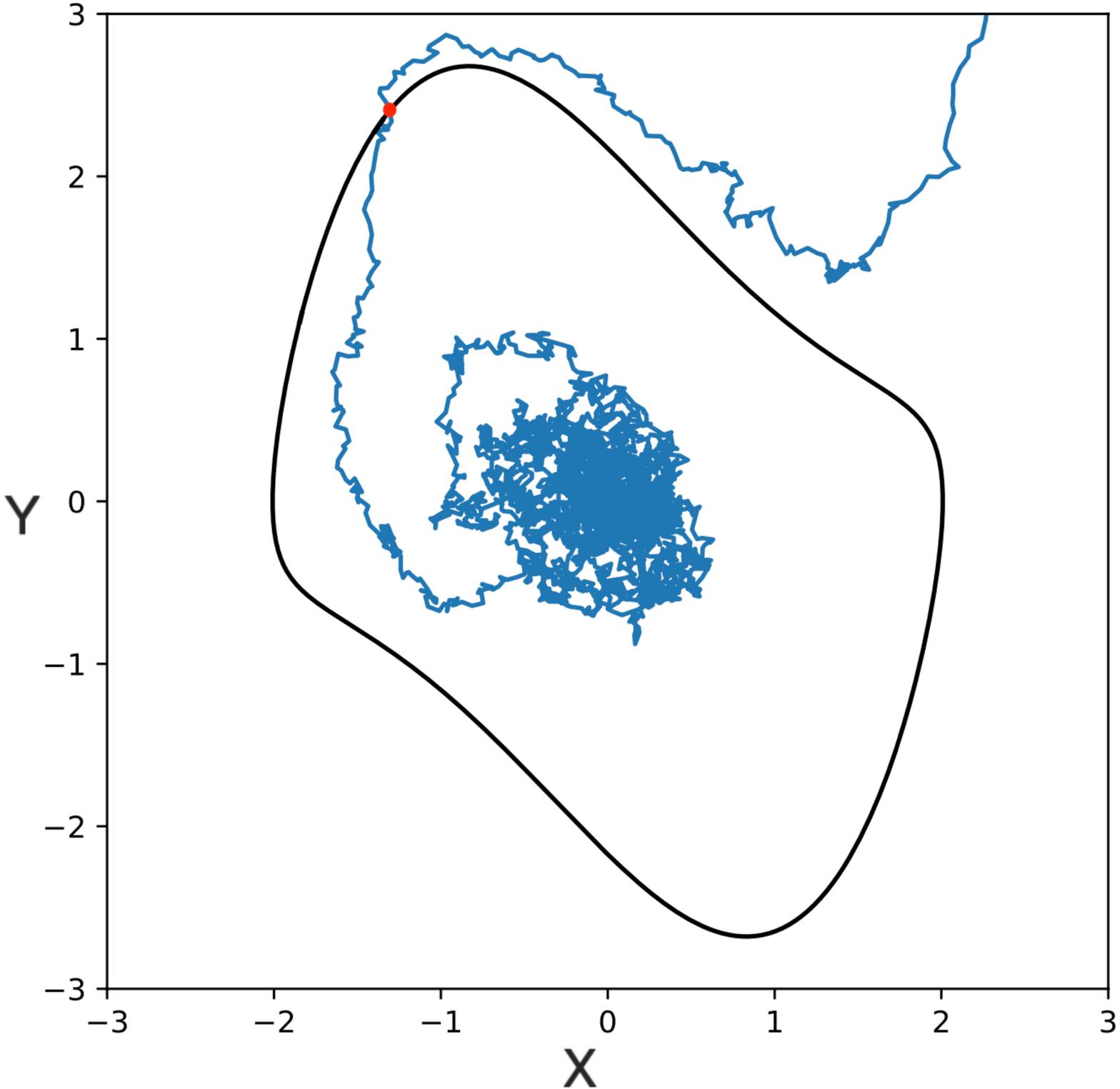}\\ (a) & (b)
\end{array}
$
\end{center}
\caption{Two sample paths of Equation \eqref{eq:first_order_ivdp_sde} (blue) on the interval $[0,200]$ with $dt=.005$, $\eta=.5$, $\sigma_1=\sigma_2=.32$, overlaid with $\Gamma$ (black). (a) The sample path does not escape. (b) The sample path escapes and the red point denotes its exit location, $(x_i,y_i)$.}
\label{fig:tipnotip}
\end{figure}

In this section, we show this by carrying out Monte-Carlo simulations on IVDP with added noise. In the computations, we have set the noise level at $\sqrt{\epsilon}=.32$, which corresponds to noise strength $\eps$ of approximately $0.1$. It may be argued that
this is not all that small, it was derived by pushing the noise to the smallest level for which we could obtain convergence on the exit distributions within reasonable computing time. We note that we find qualitatively the same results using $\sqrt{\epsilon}=.3$ or $\sqrt{\epsilon}=.35$. 

\subsection{IVDP with noise}
The stochastic version of the first order IVDP system is given by 
\begin{equation}
\begin{aligned}
dx&=y \ dt+\sqrt{\varepsilon} dW_1, \\
dy&=(-x+2\eta y(x^2-1)) dt+\sqrt{\varepsilon} dW_2.
\end{aligned}   
\label{eq:first_order_ivdp_sde}
\end{equation}
We numerically approximate the solutions of Equation \eqref{eq:first_order_ivdp_sde} using the Euler-Maruyama method to create a discretized Markov process \cite{higham_algorithmic_2001} over the time interval $[0,200].$ To apply the Euler-Maruyama method, we partition the time interval into sub-intervals of width $\Delta t=.005$, and initialize the solution at $x=0$ and $y=0$. To create the discretized Markov process, we recursively define the system as
\begin{equation}
\begin{aligned}
x_{n+1}&=x_n+y_n\Delta t + \sqrt{\epsilon} \Delta W_{1n}, \\
y_{n+1}&=y_n+(-x_n+2\eta y_n(x_n^2-1)) \Delta t + \sqrt{\epsilon} \Delta W_{2n}.
\end{aligned}  
\label{EQ:EM}
\end{equation}
A standard \correction{comment 20}{Wiener} process, $W$, satisfies the property that Brownian increments are independent and normally distributed with mean zero and variance $\Delta t$. Therefore it follows that $\Delta W_{in}=W_{in}-W_{i(n-1)}$ can be numerically simulated using $\sqrt{\Delta t} \cdot N(0,1)$. This can be shown by manipulating the probability density function of $N(0,\Delta t)$.

\subsection{The Algorithm} We want to find the realizations that have transitioned from the origin to somewhere outside the unstable periodic orbit, and capture where on $\Gamma$ they have exited. Let $\tau_i$ denote the first time a path, $X_i$, crosses $\Gamma$. We define escape events to be the paths $X_i$ that have $\tau_i \leq 200$. Let the point of $X_i$ at $\tau_i$ be given by $(x_i,y_i)$. Refer to Figure \ref{fig:tipnotip} for an example of realizations that have and have not escaped on the finite time interval. Assume for $N$ realizations there are $K$ escape events. We construct the distribution for the $x$ and $y$ locations for the $K$ escape events. To verify we have a converged result for the distribution of the location of escape events, we use the following process:
\begin{enumerate}
\item Bin the $x$ (respectively $y$) locations of the $K$ escape events by the Freedman Diaconis rule \cite{freedman_histogram_1981}. This separates the $K$ escape events into $B$ bins of equal length. 
\item Run another $N$ realizations of Equation \eqref{eq:first_order_ivdp_sde} on the same time interval and with the same step size. Assume there are $J$ escape events. We bin the $J$ escape events by the same number of bins $B$ found in Step 1. 
\item There are now two vectors of the same length, $D_1$ and $D_2$, where each component of the vector represents the amount of paths that tipped in that interval for the $x$ (respectively $y$) location. Calculate $Err=\frac{||D_1-D_2||_2}{||D_1||_2}$, which is the relative error between the two data sets. 
\item If $Err<0.1$, we say we have found the converged distribution. However, if $Err \geq 0.1$, we double the number of samples and repeat this process.
\end{enumerate}
In addition, we use the Kolmogorov-Smirnov Two Sample Test \cite{noauthor_kolmogorovsmirnov_2008} as a final verification that we have a converged distribution.

\subsection{The Escaping Paths}

We study Equation \eqref{eq:first_order_ivdp_sde} with $\eta=.5$ As mentioned above, we find  the same results if we use $\sqrt{\epsilon}=.3$ or $\sqrt{\epsilon}=.35$. and $\sqrt{\epsilon}=.32$. In these noise regimes, on the interval $[0,200]$, initialized at the origin with a step size of $dt=.005$, we find the percentage of samples that escape to be approximately [2\%,5.5\%,17.5\%] for $\sqrt{\epsilon}=[.3,.32,.35]$. We focus on $\sqrt{\epsilon}=.32$ as it is the smallest noise we can study without too much computational stress.

Using the process outlined above, we find converged distributions for exit location in both $x$ and $y$ along $\Gamma$ for this noise regime. We started with two sets of $N=50000$ realizations, doubled it to two sets of $N=100000$, and then doubled again to two sets of $N=200000$. In this case, $Err_x=0.055<0.1$ and $Err_y=0.069<0.1$. In total, there are 400000 simulations in which 21801 realizations escape.

Collecting the points $(x_i,y_i)$ from the paths that escaped, we see that they fall on specific parts of $\Gamma$. In Figure \ref{fig:MPEPC5} (b) and (c), we plot both a heatmap and jointplot of the exit locations respectively and see that there are two distinct spots on the $\Gamma$ where trajectories mostly exit. Additionally, we see the symmetry of exit locations along $\Gamma$. 

\section{Matching Simulations and Theory
\label{sec:conclusion}}

The point here is to compare our theoretical prediction with the Monte Carlo simulations and thus confirm our mathematical derivation of the MPEP. The main takeaway is the connection between the pivot point, the OM-selected point and the exit distribution. 

\begin{figure}[tbp]
\begin{center}
$
\begin{array}{ccc}
\includegraphics[scale=0.2]{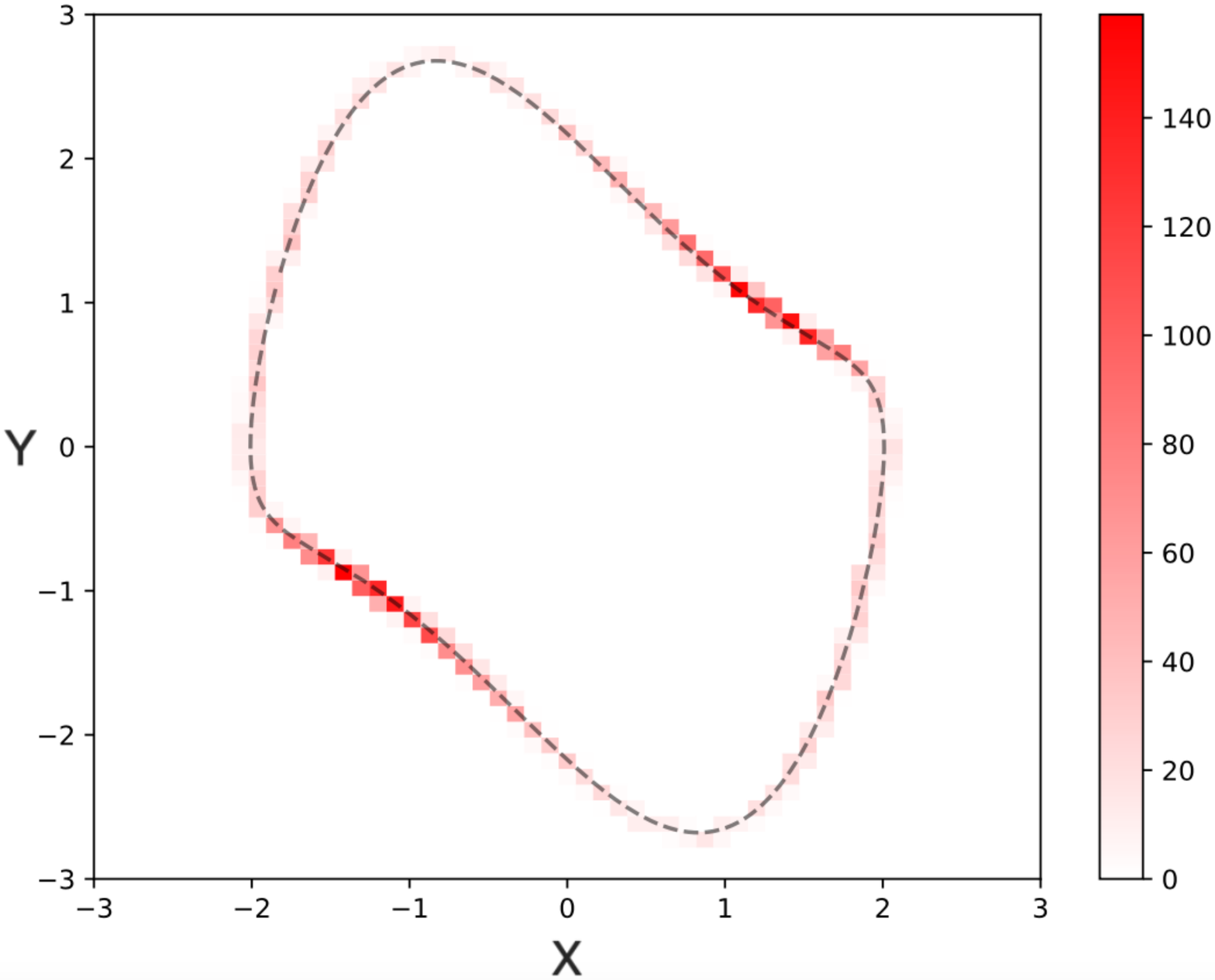} & \hspace{-7mm}\includegraphics[scale=0.24]{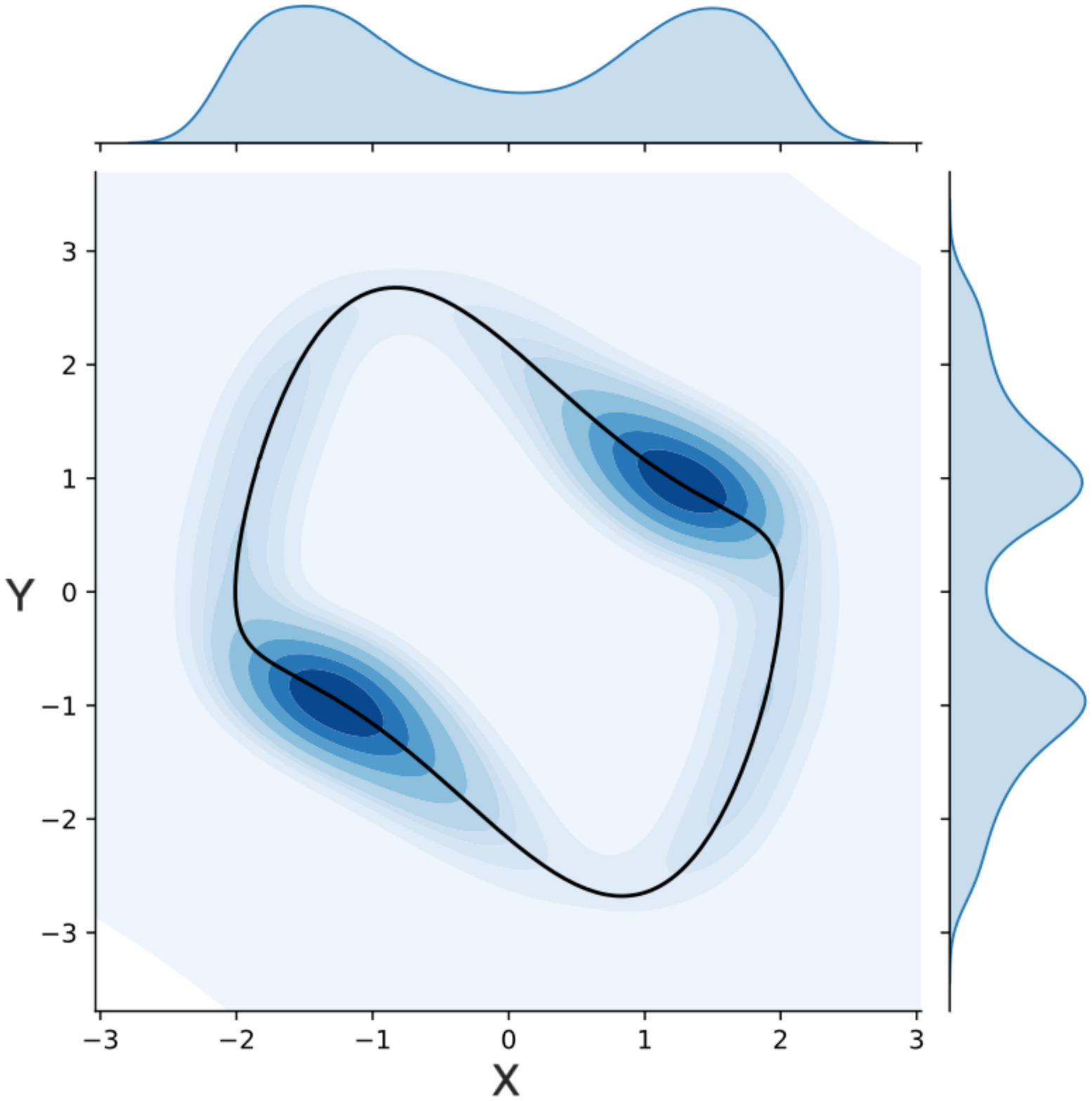} \\ (a) & (b)
\end{array}
$
\end{center}

\caption{Parameters are set at $\eta=.5, \sqrt{\epsilon}=0.32$: (a) Heatmap of points $(x_i,y_i)$ from 21801 realizations, (b) Jointplot of points $(x_i,y_i)$ from 21801 realizations. In both figures (a) and (b), we can clearly see two clear regions of exit points.}
\label{fig:MPEPC5}
\end{figure}

\subsection{The Escape Hatch and the Pivot Point}

In Section \ref{sec:pivot}, we define a subset of the mouth of the River, resulting from trajectories with Maslov Index 0, as a set $Q$. We further define the {\em pivot point} from the mouth of the river where the associated trajectory will pick up a conjugate point exactly on $\mathcal{T}_{\Gamma}$ and delineates $Q$ on one end. However, the set $Q$ does not pick out any particular part of the periodic orbit since its projection onto the $(x,y)$-space is all of $\Gamma$. Nevertheless, the exit distribution from Section \ref{sec:mcsim} does pick out particular parts of $\Gamma$. 

\begin{figure}[tbp]
\begin{center}
$
\begin{array}{cc}
\includegraphics[scale=0.25]{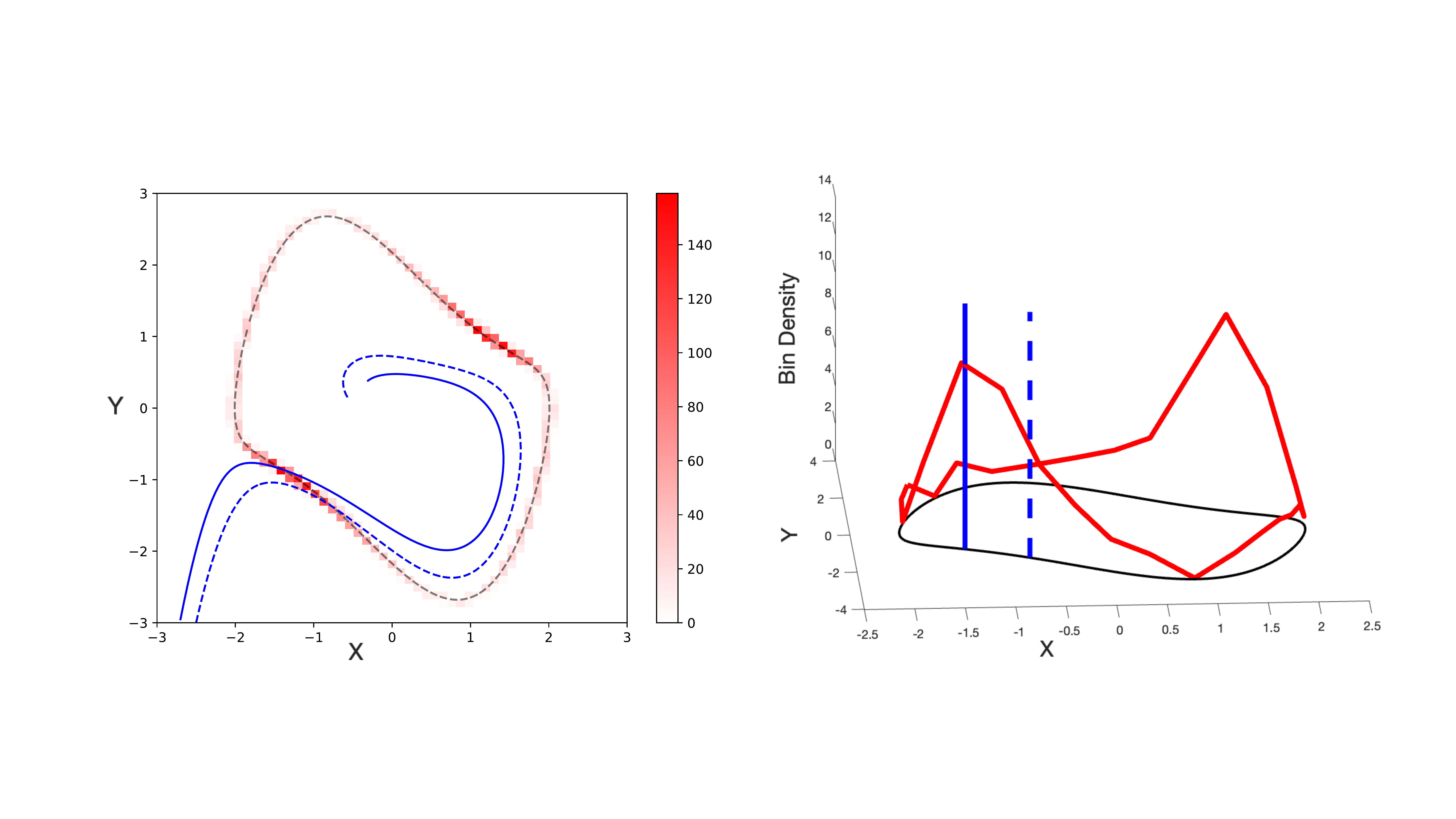}  &
 \includegraphics[scale=0.25]{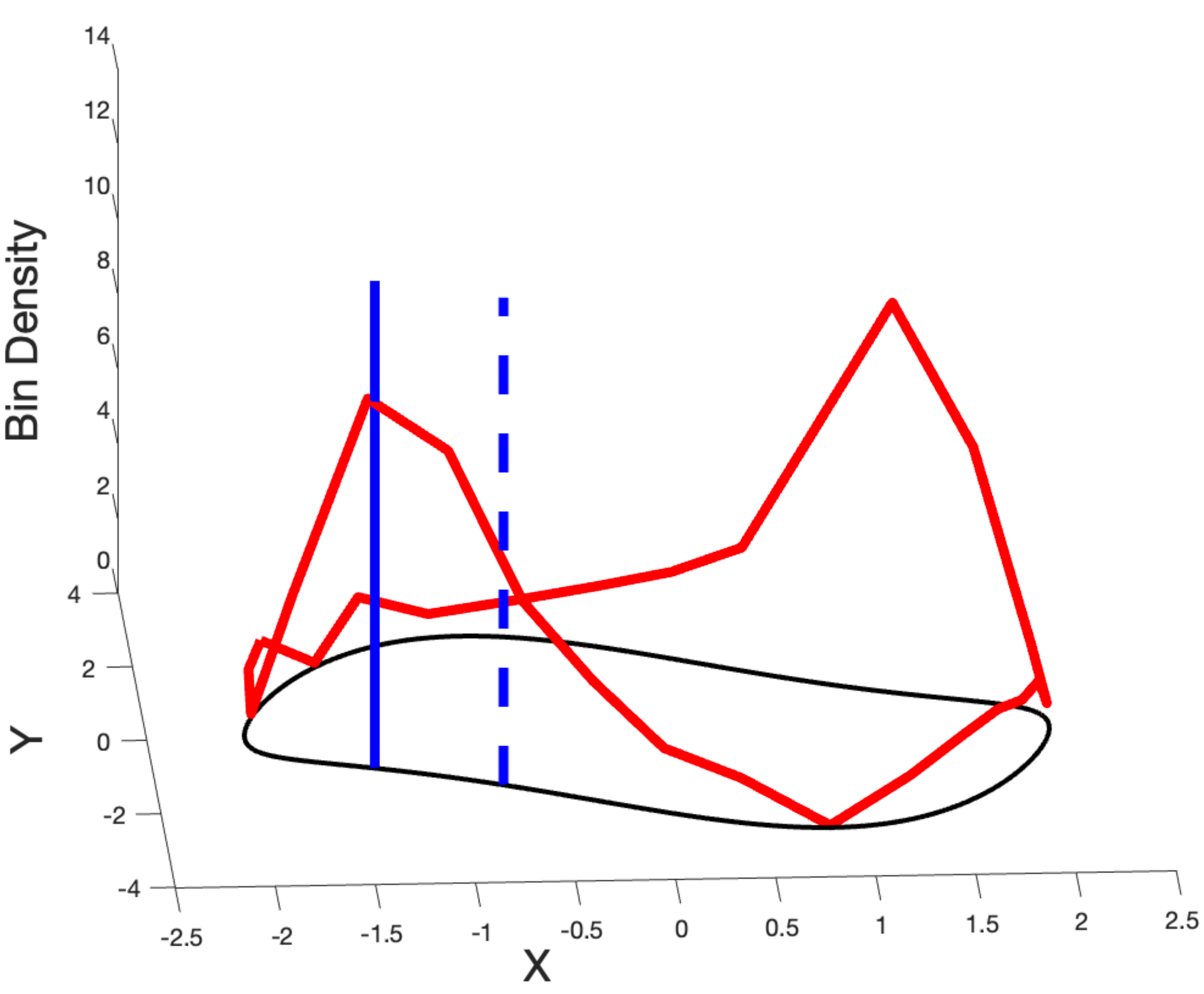}\\ (a) & (b)
\end{array}
$
\end{center}
\caption{In (a) the solid blue curve marks the predicted most probable escape path as determined by minimizing the action of the OM functional, computed as a perturbation of the FW functional. The dashed blue curve depicts the trajectory that corresponds to the pivot point as describe in Section \ref{sec:pivot}. The heat map displayed on the periodic orbit depicts the probability of an escape event occurring at that location, as described in Section \ref{sec:mcsim}.  The red curve in Figure (b) provides a 3D representation of the heat map displayed in Figure (a), and the vertical blue lines indicate the intersection of the orbits displayed in Figure (a) with the periodic orbit.  }
\label{fig:MPEPC2}
\end{figure}

Surprisingly, the trajectories choose the initial part of $Q$ for their escape. Figure \ref{fig:MPEPC5} shows the dominating part of the escape hatch through two different kinds of heatmap. They are clearly in the southwest and northeast corners of the periodic orbit. Recalling the symmetry, we can focus on one part, and we choose the southwest corner. 

In panel (a) of Figure \ref{fig:MPEPC2}, the dashed blue curve represents the trajectory of Equation \eqref{eq:hamil} that exits at the pivot point. This is seen to be at the right hand end of the escape hatch as determined by the Monte-Carlo simulations. Most of the trajectories clearly exit beyond the pivot point in terms of the natural ordering on $Q$. But they exit relatively close to it rather than continuing to follow the unstable manifold along $\Gamma$ and exiting further later, after which the action would have actually decreased.

\subsection{The Escape Hatch and the OM point}
The Onsager-Machlup functional shows why the noisy trajectories choose to exit in a region only just beyond the pivot point. In Figure \ref{fig:MPEPC2}, the solid blue curve was computed from the minimum of the  OM functional along FW orbits. This is the selection mechanism we have discussed and we claim justifies the designation of the OM-trajectory as the MPEP for the associated level of noise. The angle is found from the graph in Figure \ref{fig135}. The minimum occurs around $\theta =0.44$, which value is used to initiate the trajectory on $\mathcal{K}$, and we call this the OM-trajectory. The OM-selected point is the point where this trajectory crosses the periodic orbit $\Gamma$. 

The OM selected trajectory is shown as a solid blue curve in Panel (a) of Figure \ref{fig:MPEPC2}. The OM-point is represented by the solid, vertical blue line in Panel (b). Fom Panel (b) the OM-point can be seen to coincide with the peak of the exit distribution. 
\begin{figure}[tbp]
\begin{center}
\includegraphics[height=7cm]{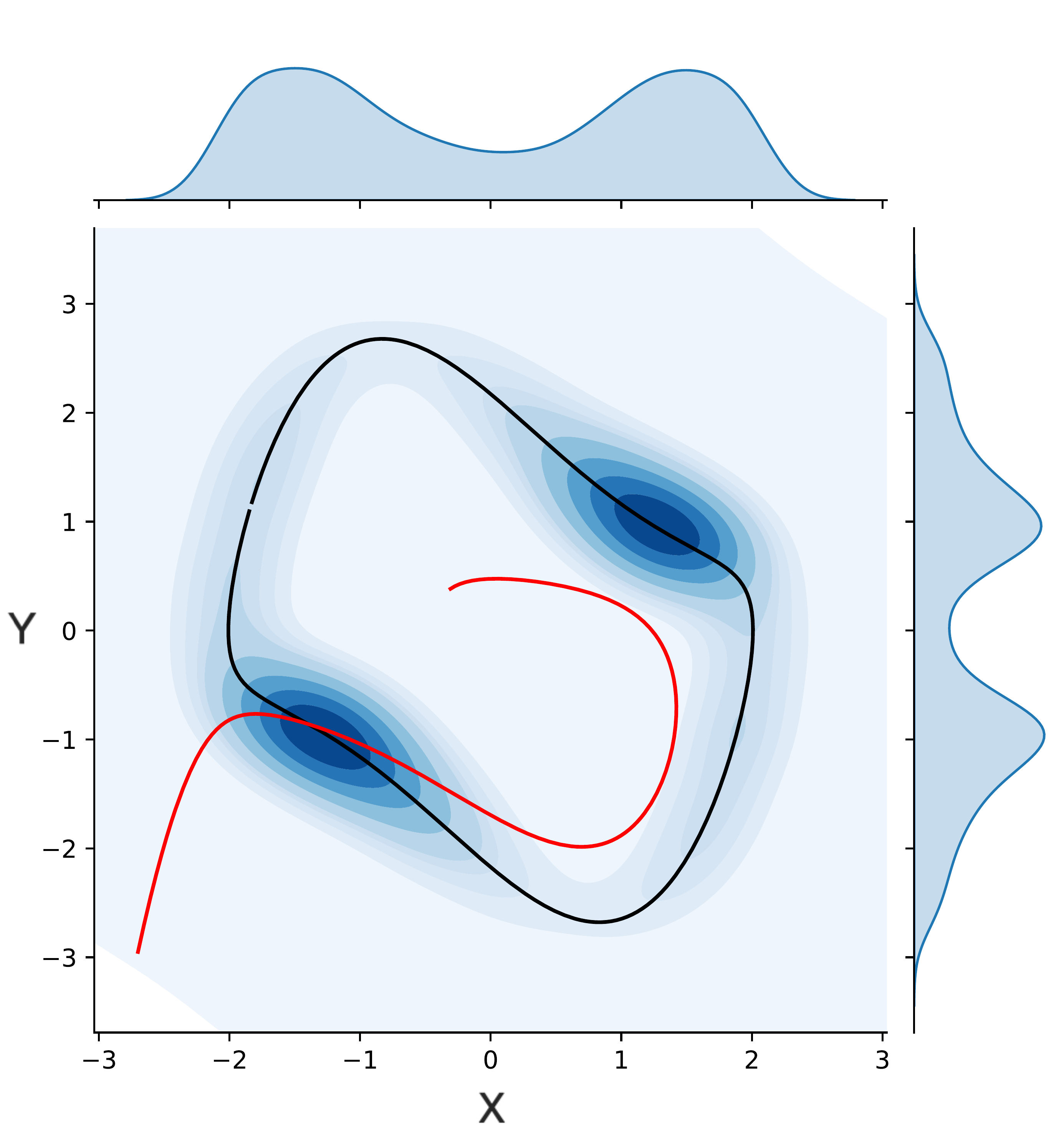}
\end{center}
\caption{The solid red curve marks the predicted most probable escape path as determined by minimizing the action of the OM functional, computed as a perturbation of the FW functional. The blue heat map depicts the joint plot of the exit locations of noisy paths on the periodic orbit. }
\label{fig:MPEPC3}
\end{figure}

Panel (a) of Figure \ref{fig:MPEPC2} renders the entire trajectory, whereas Panel (b) focuses on the exit set on $\Gamma$. The vertical dashed and solid lines give the location of the pivot point and OM-selected point respectively. Their relationship with the exit distribution is self-evident. The pivot point pins one end of the distribution, while the OM selected point lies at the peak of the distribution. We have not found a specific characterization of the left end of the distribution, but it does appear to drop off rapidly after the OM point. The distribution on the pivot point side has a much gentler drop-off. Note that this is reminiscent of the Gumbel-type distributions often seen in these situations.

The significance of the OM point as being at the peak of the exit distribution is depicted further in Figure \ref{fig:MPEPC3}. This figure shows the striking coincidence of the center of the heatmap with the OM point. Note that the pivot point is independent of the noise as it only depends on the FW functional. On the other hand, the OM point depends on the noise as it is based on the OM functional. In this case, it is evaluated with the same level of noise as we use for the Monte-Carlo simulations. We anticipate that, as noise is reduced, it would move around the periodic orbit.

The exit points themselves on the periodic orbit $\Gamma$ have been emphasized so far. We can compare the full trajectories with the Monte-Carlo simulations to see that the noisy trajectories do indeed closely follow the OM selected trajectory. 
In Figure \ref{fig:MPEP_KDE}, an estimate of the time slices of the noisy trajectories that exit is shown as a series of dots. 
\begin{figure}[tbp]
\begin{center}
\includegraphics[height=6cm]{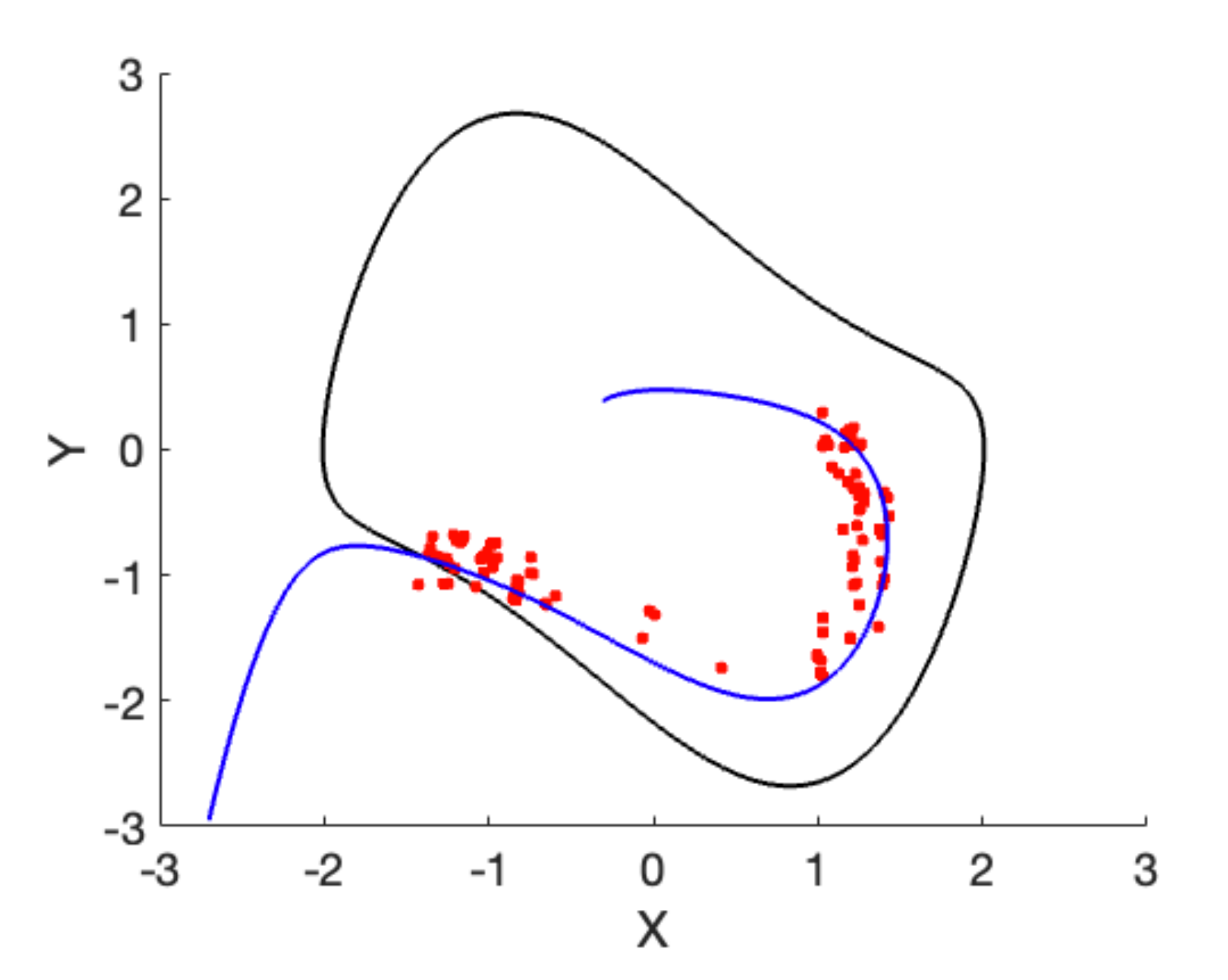}
\end{center}
\caption{The OM-trajectory is depicted as a blue curve and the red dots represent the values of the kernel density estimates of the time slice distributions of exiting, noisy trajectories.}
\label{fig:MPEP_KDE}
\end{figure}
A kernel density estimate is used to estimate the time slice distributions. The trajectories are reparameterized to begin on a given circle around the origin. This circle is chosen large enough so as to make the different future time slices of the family of noisy trajectories comparable. While there is some arbitrariness in this choice of time parameterization, it gives an appropriate picture of the time evolution of the distribution of noisy trajectories. The OM-trajectory is depicted again as the solid blue curve and it is seen to give a fairly good approximation of the time slices. 

The most important point to take away from Figure \ref{fig:MPEP_KDE} is not just that the OM-point and the peak of the exit distribution match on $\Gamma$, but that the OM-trajectory is matching the distribution of noisy trajectories all along the path. Our interpretation is that the OM-trajectory is acting as a guide for the exiting trajectories of the stochastic system. 

\correction{comment 9}{Furthermore, as the noise strength $\varepsilon$ becomes smaller, the theory of Day \cite{day_exit_1996} suggests that the escape hatch would smoothly rotate around the periodic orbit. We believe that it in fact jumps to another weak part of the periodic orbit. Due to the symmetry in the IVDP problem, this will be approximately one-half period around the orbit. This jumping behavior will be repeated when $\varepsilon$ is reduced further. }

\begin{figure}[htbp]
 \begin{center}
$
\begin{array}{lcr}
(a) \includegraphics[scale=0.22]{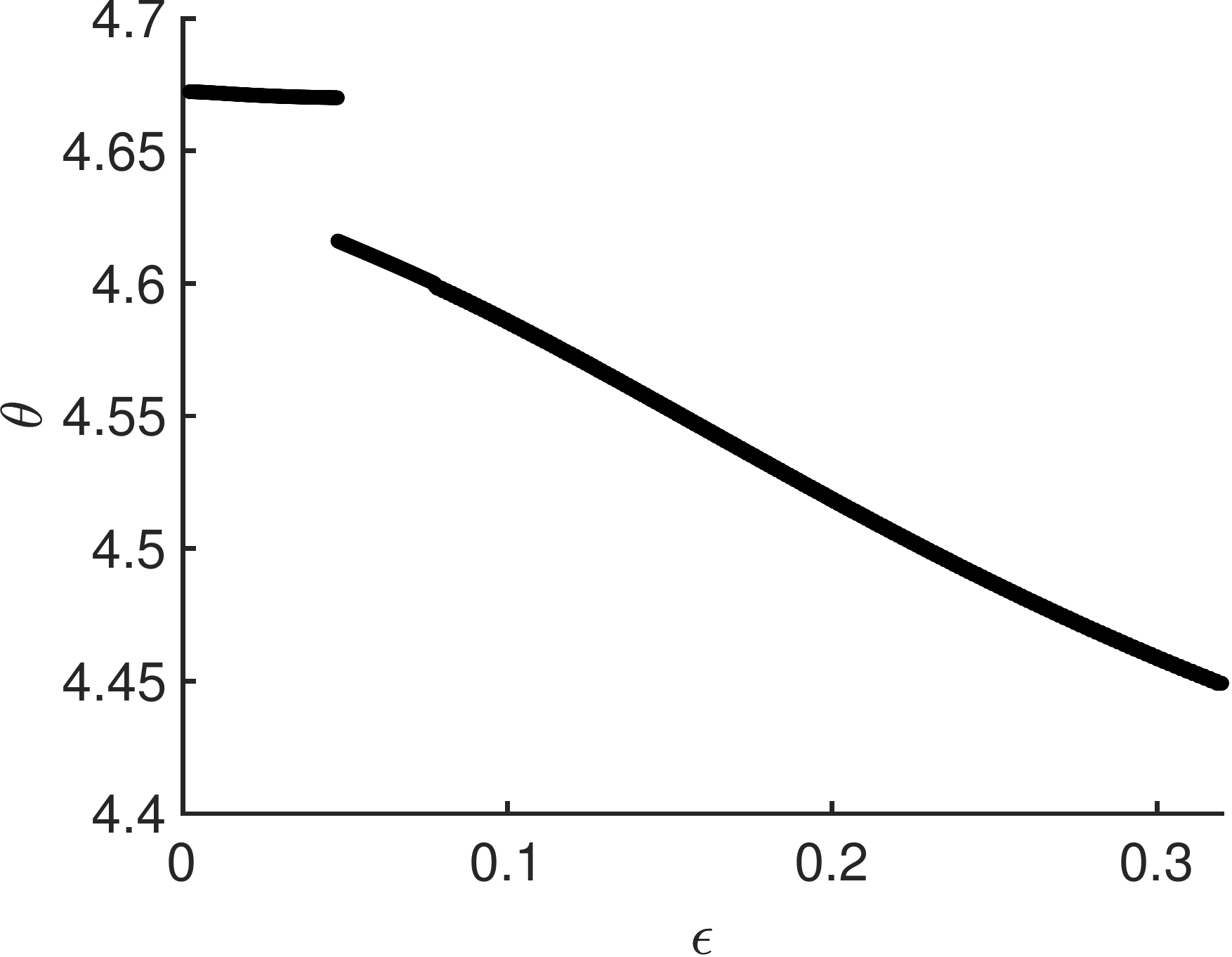} & (b) \includegraphics[scale=0.22]{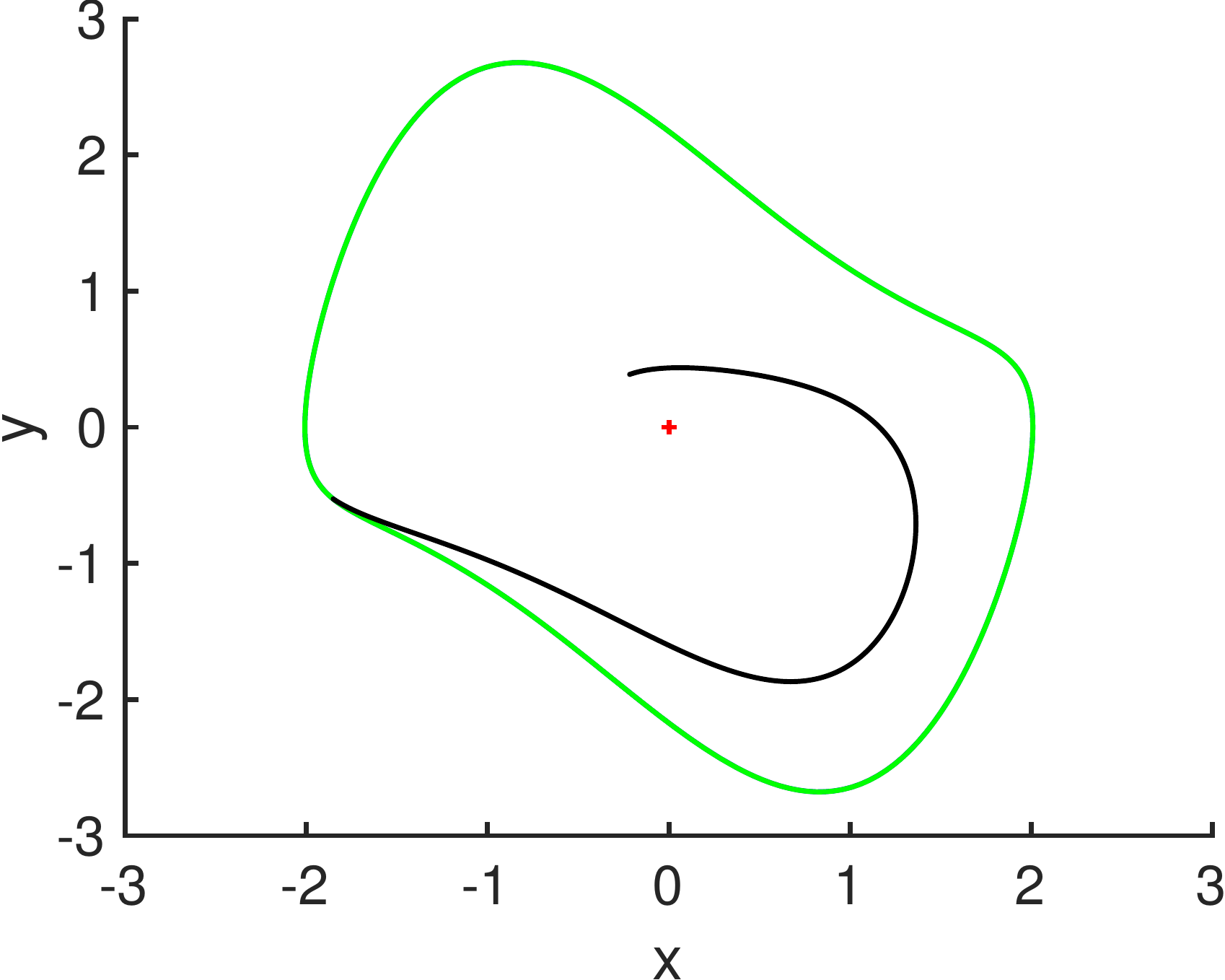}& (c) \includegraphics[scale=0.22]{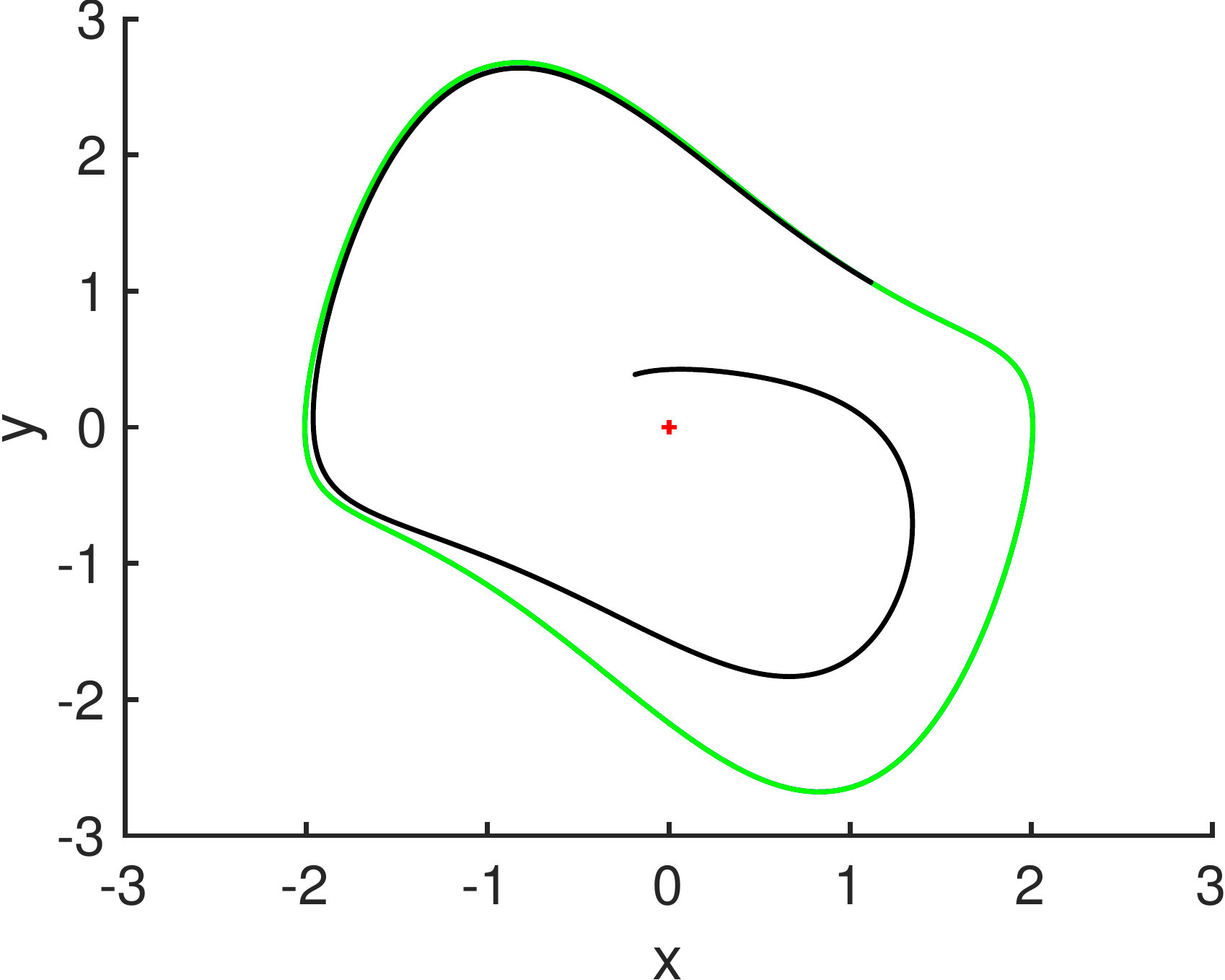}
\\
(d) \includegraphics[scale=0.22]{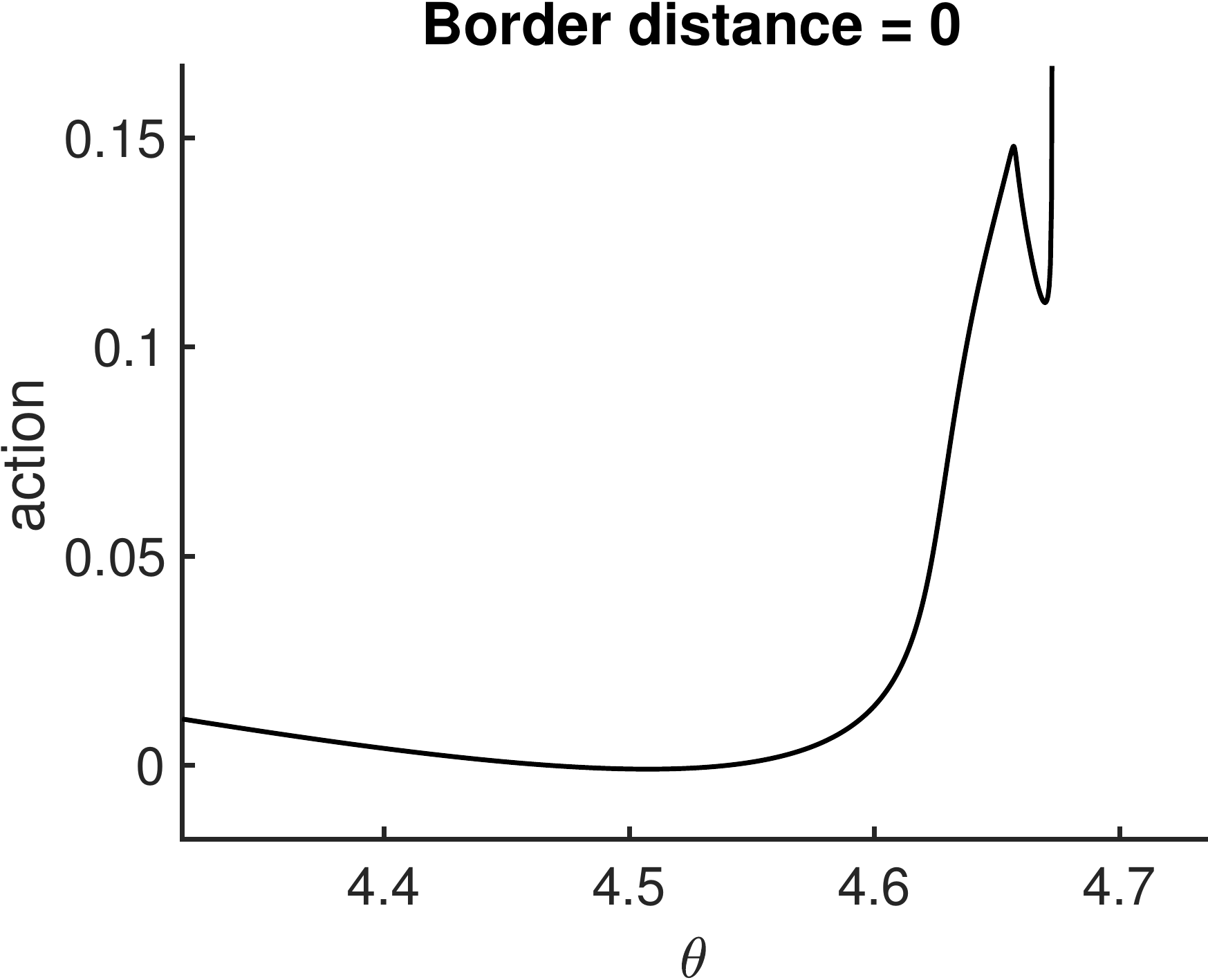} & (e) \includegraphics[scale=0.22]{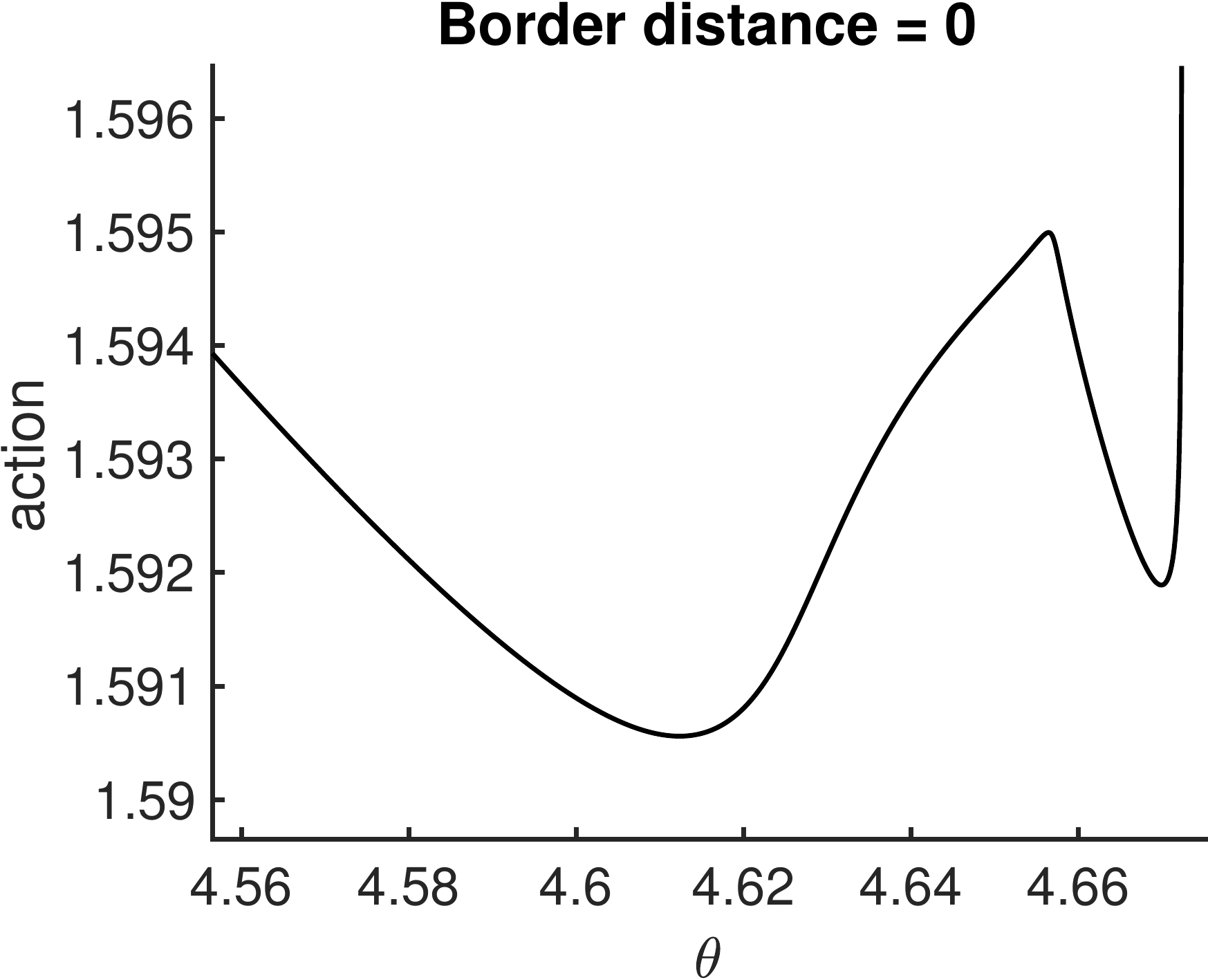}
  &(f) \includegraphics[scale=0.22]{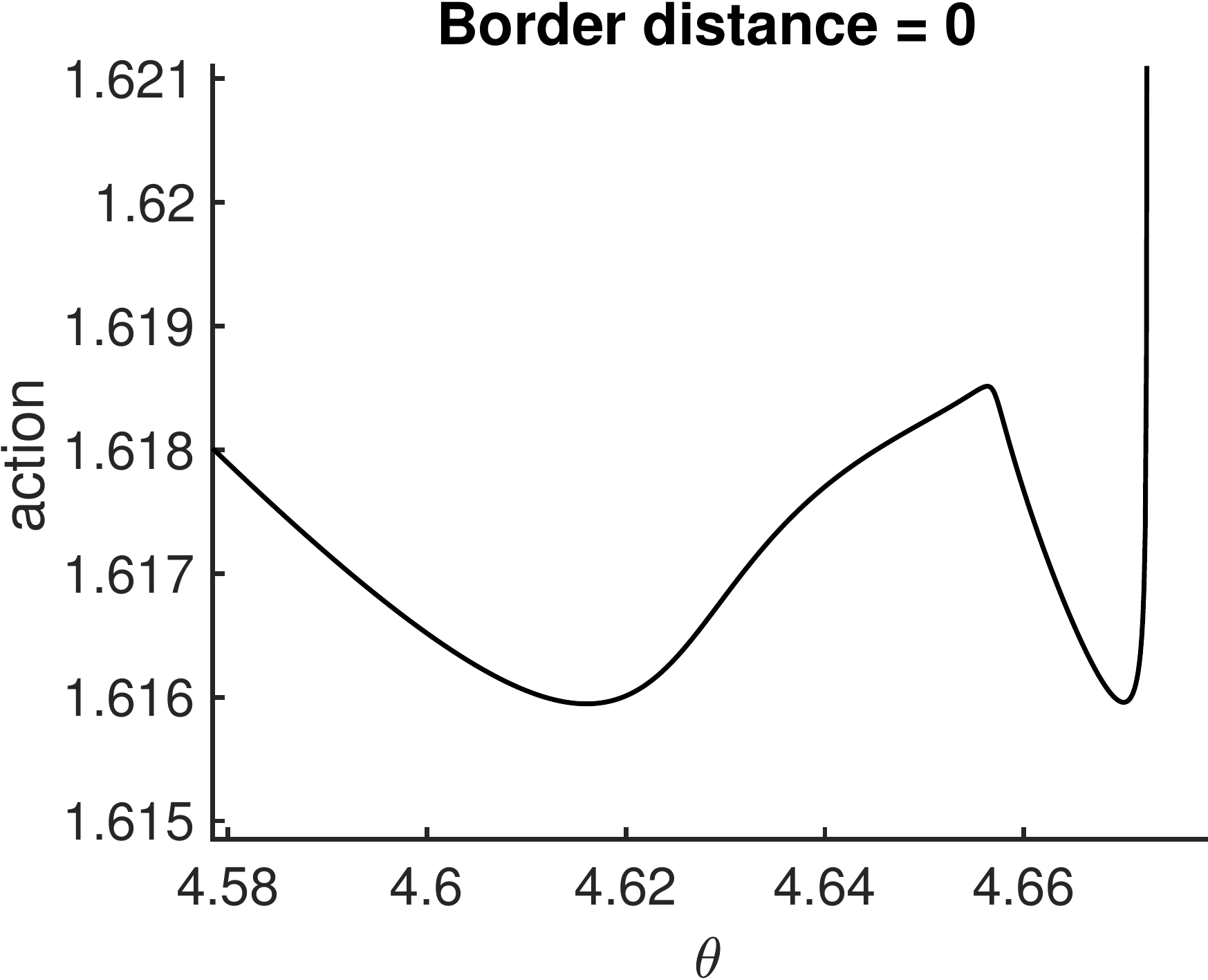}
\end{array}
$
\end{center}
\caption{(a) Plot of the value of $\theta$ that corresponds to a global minimum of the Onsager-Machlup functional as determined as a perturbation of the Friedlin-Wentzell functional, against the noise strength $\eps$. (b) Plot of the global minimizer corresponding to $\eps = 4.7595\mathrm{e}{-2}$, which corresponds to  $\theta =  4.616014$. (c) Plot of the global minimizer corresponding to $\eps = 2.2753\mathrm{e}{-3}$, which corresponds to $\theta = 4.670036$. (d)-(f) Plot of the action of the Onsager-Machlup functional against $\theta$ for $\eps = 4.7595\mathrm{e}{-2}$, $\eps = 3.000\mathrm{e}{-3}$, and $\eps = 2.2753\mathrm{e}{-3}$, respectively. }
\label{fig141}
\end{figure}

\correction{comment 10}{This effect can be seen from the use of the OM functional. Indeed, the OM point does jump, in a discontinuous fashion, to the other side of the periodic orbit. This can be seen in panels (a)-(c) of Figure \ref{fig141}. Panel (a) gives the $\theta$ value on the initiating circle for the OM point (note that {\it Border distance} is from the periodic orbit so that it being 0 means that we are minimizing the functional up to where they cross $\Gamma$). At a certain value of $\varepsilon$ there is a jump in the $\theta$-value of the OM point and the corresponding trajectories are seen to change from that shown to panel (b)-to the right of the jump-to that shown in panel (c) to the left, i.e., when $\varepsilon$ is smaller. Panels (d)-(e) show why this occurs by looking at the value of the OM functional along the relevant FW orbits. There is a local minimum to the right of the absolute minimum, which corresponds to the OM point discussed in this paper. As $\varepsilon$ is decreased, the value at this local minimum decreases and takes over as the absolute minimum at a certain value of $\varepsilon$. we anticipate that there would be further minima to the right and that these would correspond to further cycling.}

\section{ Conclusion and Discussion
\label{sec:conclusion1}}

We have developed a dynamical systems approach for computing most probable escape paths where the boundary of the basin of attraction is a periodic orbit, and the noise strength is small but not vanishingly so. The key is the isolation of a subset of the unstable manifold of the equilibrium solution surrounded by the periodic orbit, which we call the River. This subset of the unstable manifold is delineated by heteroclinic orbits which connect the equilibrium solution to the periodic orbit. We use the Maslov index to distinguish local minimizers (subject to a fixed boundary condition) from other extremizing orbits. In addition, we establish a connection between the folding of $W^u(O)$ and the appearance of conjugate points along its trajectories. 

Much previous work has been done in studying MPEPs over periodic boundaries. In \cite{beri_solution_2005}, the authors studied the structure of the escape trajectories and showed that the Most Probable Escape Path (MPEP) reaches the limit cycle asymptotically with no momentum. In \cite{e_study_2010}, the authors also noted that in the case of an unstable limit cycle coexisting with a stable fixed point, the MPEP spirals toward the limit cycle asymptotically and its $\omega$-limit set is the complete limit cycle; \cite{maier_oscillatory_1996} showed that the MPEP does indeed reach the limit cycle asympotically and trajectories exiting are necessarily, optimal trajectories that are small perturbations of the MPEP. 

For intermediate noise regimes, the cycling is hardly evident and a specific subset of the boundary appears to be chosen through which the primary leakage of the escaping paths occurs. Our work is aimed at providing a theoretical underpinning for this phenomenon.

The core elements of the methodology can be summarized as follows:
\begin{enumerate}
    \item Use the 4D Hamiltonian system derived from the Euler-Lagrange equations from the Friedlin-Wentzell functional to compute stable and unstable manifolds for the periodic orbit in $H=0$ and the equilibrium solution respectively.
    \item Compute the heteroclinic orbits that arise from the transverse intersections of those invariant sets.
    \item Identify a set of trajectories delineated by the heteroclinic orbits where the unstable manifold of the equilibrium solutions leaks out of the periodic orbit (when projected on the $(x,y)$-space). We call this the River.
    \item Use the Maslov index to weed out the trajectories that do not correspond to local minimizers.
    \item Find the end-point of the set of trajectories with Maslov Index 0. This is the pivot point and is characterized by having a conjugate point exactly when crossing the periodic orbit.  
    \item Compute the action using the Onsager-Machlup functional as a pertubation to the Friedlin-Wentzell functional for trajectories in the part of the river found in the previous step.
    \item Use the OM trough (global minimum) to compute the associated trajectory in the 4D Hamitonian system originally derived from the Friedlin-Wentzell functional. This we call the OM-trajectory and is the MPEP for the given level of noise.
    \item Verify that this trajectory has no conjugate point before hitting the boundary.
    \item Finally, superimpose these trajectories on the converged result for the distribution of the location of escape events on the periodic orbit in order to validate our computations.
\end{enumerate}

For the IVDP, we carried out this program and showed a striking correspondence between the exit distribution and the OM-trajectory. Moreover, the pivot point acts as an anchor for the exit distribution with the exit set of the noisy trajectories choosing a region not much beyond it. 

Considerable insight can be gained from taking this dynamical systems perspective. The phenomenon in which parts of the unstable manifold of the fixed point cross the periodic does not occur in gradient systems and is a consequence of the transverse intersection of a stable (for the periodic orbit) and an unstable (for the fixed point) invariant manifold. 

Since the Freidlin-Wentzell functional is independent of noise, these dynamical constructions do not depend on the noise strength. Nevertheless, Large Deviation Theory can only be invoked to see how the Euler-Lagrange equations guide the noisy trajectories in the limit of vanishing noise. In our case, that theory predicts cycling. Our objective was to use the theoretical constructs of Freidlin-Wentzell theory to show how cycling is actually resisted when noise is made slightly larger.

\section*{Acknowledgement}
The authors wish to thank John Gemmer for helpful conversations. The authors also wish to thank the anonymous reviewers for very helpful comments that have led to a much-improved version of the paper. In particular, the issue raised at the end about the possibility of the escape hatch jumping as $\varepsilon$ tends to zero was based on a very insightful question from one of the reviewers. Emmanuel Fleurantin was supported by NSF grant DMS-2137947 and Office of Naval Research grant N000141812204 during the work on this research. Christopher Jones and Katherine Slyman were supported by Office of Naval Research grant N000141812204.

\bibliographystyle{unsrt}
\bibliography{refs}

\medskip

\end{document}